\documentclass[11pt]{amsart}
\usepackage[T1]{fontenc}
\usepackage{amssymb,amsmath, amsthm, amsfonts, tikz-cd}

\usepackage{graphicx}
\usepackage{listings}
\usepackage[margin=1in]{geometry}
\usepackage{lstautogobble}
\usepackage{enumerate}
\usepackage[shortlabels]{enumitem}
\usepackage{thmtools}
\usepackage{thm-restate}
\usepackage{amsthm}
\usepackage{verbatim}

\usepackage{mathtools}
\usepackage{physics}
\usepackage{color}
\usepackage{hyperref}
\usepackage[capitalise]{cleveref}
\usepackage[bottom]{footmisc}
\crefformat{equation}{(#2#1#3)}

\usepackage{float}
\restylefloat{table}

\usepackage{mathrsfs}
\setlist{  
  listparindent=\parindent,
  parsep=0pt,
}

\theoremstyle{plain}
\newtheorem{thm}{Theorem}[section]
\newtheorem{prop}[thm]{Proposition}

\theoremstyle{definition}

\newtheorem{remark}[thm]{Remark}

\Crefname{thm}{Theorem}{Theorems}
\Crefname{prop}{Proposition}{Propositions}

\numberwithin{equation}{section} 

\DeclarePairedDelimiter{\paren}{\lparen}{\rparen}

\DeclarePairedDelimiter{\jp}{\langle}{\rangle}

\DeclareMathOperator{\supp}{supp}

\newcommand{\M}{{\mathcal{M}}}

\newcommand{\p}{{\partial}}

\renewcommand{\d}{\delta}

\newcommand{\R}{{\mathbb{R}}}

\newcommand{\N}{{\mathbb{N}}}

\newcommand{\Z}{{\mathbb{Z}}}

\newcommand{\Ss}{{\mathbb{S}}}
\renewcommand{\H}{{\mathcal{H}}}
\renewcommand{\P}{{\mathcal{P}}}
\newcommand{\T}{{\mathbb{T}}}

\newcommand{\G}{{\mathfrak{G}}}

\newcommand{\Fr}{F}

\renewcommand{\L}{{\mathcal{L}}}

\renewcommand{\M}{{\mathcal{M}}}

\newcommand{\Dc}{\mathcal{D}}

\newcommand{\tl}{\tilde}

\newcommand{\D}{\Delta}
\newcommand{\Uu}{\mathfrak{U}}

\newcommand{\nn}{\nonumber}

\newcommand{\ol}{\overline}
\newcommand{\ul}{\underline}

\newcommand{\ux}{\underline{x}}

\newcommand{\ep}{\epsilon}
\newcommand{\vep}{\varepsilon}
\newcommand{\al}{\alpha}
\newcommand{\be}{\beta}

\newcommand{\Ib}{\mathbb{I}}
\renewcommand{\G}{\mathsf{G}}
\renewcommand{\Fr}{\mathsf{F}}

\newcommand{\indic}{\mathbf{1}}
\newcommand{\la}{\lambda}

\let\div\relax
\DeclareMathOperator{\div}{div}

\let\oldtocsection=\tocsection
 
\let\oldtocsubsection=\tocsubsection
 
\let\oldtocsubsubsection=\tocsubsubsection
 
\renewcommand{\tocsection}[2]{\hspace{0em}\oldtocsection{#1}{#2}}
\renewcommand{\tocsubsection}[2]{\hspace{1em}\oldtocsubsection{#1}{#2}}
\renewcommand{\tocsubsubsection}[2]{\hspace{2em}\oldtocsubsubsection{#1}{#2}}

\begin{document}

\title[From Quantum Many-Body Systems to Ideal Fluids]{From Quantum Many-Body Systems to Ideal Fluids}

\author[M. Rosenzweig]{Matthew Rosenzweig}
\email{mrosenzw@mit.edu}
\thanks{M.R. is supported by the Simons Foundations through the Simons Collaboration on Wave Turbulence and by the NSF through grant DMS-2052651.}

\begin{abstract}
We give a rigorous, quantitative derivation of the \emph{incompressible Euler equation} from the many-body problem for $N$ bosons on $\T^d$ with binary Coulomb interactions in the semiclassical regime. The coupling constant of the repulsive interaction potential is $~1/(\vep^2 N)$, where $\vep \ll 1$ and $N\gg 1$, so that by choosing $\vep=N^{-\lambda}$, for appropriate $\lambda>0$, the scaling is \emph{supercritical} with respect to the usual mean-field regime. For approximately monokinetic initial states with nearly uniform density, we show that the density of the first marginal converges to 1 as $N\rightarrow\infty$ and $\hbar\rightarrow 0$, while the current of the first marginal converges to a solution $u$ of the incompressible Euler equation on an interval for which the equation admits a classical solution. In dimension 2, the dependence of $\vep$ on $N$ is essentially optimal, while in dimension 3, heuristic considerations suggest our scaling is optimal. To the best of our knowledge, our result is a new connection between quantum many-body systems and ideal hydrodynamics, complementing the previously known connection to compressible fluids. Our proof is based on a Gronwall relation for a \emph{quantum modulated energy} with an appropriate corrector and is inspired by recent work of Golse and Paul \cite{GP2021} on the derivation of the pressureless Euler-Poisson equation in the classical and mean-field limits and of Han-Kwan and Iacobelli \cite{HkI2021} and the author \cite{Rosenzweig2021ne} on the derivation of the incompressible Euler equation from Newton's second law in the supercritical mean-field limit. As a byproduct of our analysis, we also derive the incompressible Euler equation from the Schr\"odinger-Poisson equation in the limit as $\hbar+\vep\rightarrow 0$, corresponding to a combined classical and quasineutral limit.
\end{abstract}
\maketitle

\section{Introduction}
The \emph{incompressible Euler equation}
\begin{equation}\label{eq:Eul}
\begin{cases}
\p_t u + u\cdot\nabla u = -\nabla p \\
\div u = 0\\
u|_{t=0} = u^0,
\end{cases}
\qquad (t,x) \in \R\times\T^d, \ d\in\{2,3\}
\end{equation}
describes the evolution of the velocity field $u$ of an ideal fluid set on the periodic domain $\T^d$ with scalar pressure $p$ and initial datum $u^0$. Formally, one can obtain equation \eqref{eq:Eul} from classical mechanics by considering the fluid as a continuum and applying Newton's second law to infinitesimal fluid volume elements. See, for instance, \cite{Tao2018}. However, turning such reasoning into a rigorous proof is difficult, and an ongoing challenge is to derive Euler's equation from more fundamental physical descriptions. This challenge is the concern of the present article, with the ultimate goal of going from the microscopic dynamics at the quantum level to the macroscopic dynamics of the fluid \eqref{eq:Eul}.

\subsection{Background}
Starting from classical mechanics, one possibility to derive Euler's equation is to go from Newton to Euler by first passing to the mesoscopic description given by Boltzmann's equation for the evolution of the distribution function in particle phase space. One can then derive the incompressible Euler equation from the Boltzmann equation in a suitable hydrodynamic scaling regime. We refer the interested reader to Saint-Raymond's monograph \cite{SR2009book} for further elaboration.

Another possibility is to obtain incompressible Euler from the \emph{Vlasov-Poisson equation}, which is a macroscopic/hydrodynamic description of Newton's second law for a system of particles with binary Coulomb interactions. Following earlier work \cite{BG1994, Grenier1995, Grenier1996,Grenier1999}, Brenier \cite{Brenier2000} (see also \cite{Masmoudi2001}) proved that the so-called \emph{quasineutral limit} leads to the incompressible Euler equation \eqref{eq:Eul}. In turn, one can then attempt to derive the Vlasov-Poisson equation from Newton's second law, by considering the position-velocity system
\begin{equation}\label{eq:New}
\begin{cases}
\dot{x}_i = v_i \\
\displaystyle\dot{v}_i = -\lambda\sum_{1\leq j\leq N : j\neq i} \nabla V(x_i-x_j),
\end{cases}
\qquad i\in\{1,\ldots,N\}
\end{equation}
in the \emph{mean-field limit} as $N\rightarrow\infty$ with coupling constant $\lambda=\frac{1}{\vep^2 N}$, where $N$ is the number of particles, $\vep>0$ is the quasineutral parameter, and $V$ is the Coulomb potential. If $\vep$ is fixed, then as $N\rightarrow\infty$, one expects the empirical measure of the system \eqref{eq:New} to converge to a solution of the Vlasov-Poisson equation. In the classical plasma physics setting of Vlasov-Poisson, the distribution function models the evolution of electrons against a stationary background of positively charged ions. An important characteristic scale of the system is the \emph{Debye (screening) length}, which is the scale at which charge separation in the plasma occurs. If the Debye length is much smaller than the length scale of observation, the plasma is called quasineutral, as it appears neutral to an observer. After a rescaling to dimensionless variables, the parameter $\vep$ acquires the meaning of the Debye length. For further comments on quasineutrality in plasmas, we refer to the surveys \cite{Manfredi2005, GpI2020qn}.

Unfortunately, going from Newton to Vlasov-Poisson for fixed $\vep>0$, let alone $\vep\rightarrow 0^+$, is not fully understood. While the mean-field limit of \eqref{eq:New} with regular $V$ (e.g. globally $C^{1,1}$) \cite{NW1974, BH1977, Dobrushin1979, Duerinckx2021gl} or even just bounded force $\nabla V$ \cite{JW2016} has been shown, the Coulomb case remains out of reach except in dimension 1 \cite{Hauray2014}. The best results for singular potentials are limited to the strictly sub-Coulombic case \cite{HJ2007, HJ2015} or are for Coulomb potentials with short-distance vanishing cutoff \cite{BP2016, Lazarovici2016, LP2017, Grass2021}. Only recently, has the Coulomb case---and even the super-Coulombic Riesz case---come in to reach for \emph{monokinetic} initial data, for which the Vlasov-Poisson equation reduces to the \emph{pressureless Euler-Poisson equation}, due to work of Duerinckx and Serfaty \cite[Appendix]{Serfaty2020}. 

Recently, Han-Kwan and Iacobelli \cite{HkI2021} rigorously derived Euler's equation \eqref{eq:Eul} directly from the Newtonian $N$-body problem \eqref{eq:New} in the combined mean-field and quasineutral limit, assuming monokinetic initial data and imposing some restrictions on how slowly the coupling constant can vanish. The empirical measure of the system \eqref{eq:New} converges to a probability measure on $\T^d\times\R^d$ which is uniform in space and concentrated in velocity on the solution to \eqref{eq:Eul}. Writing $\vep^2 N = N^{\theta}$ for suitable $0<\theta <1$, this regime has the attractive interpretation of a \emph{supercritical mean-field limit}---an interpretation we adopt in this paper. The terminology coined by Han-Kwan and Iacobelli stems from the fact that the force term in \eqref{eq:New} is more singular than in the usual mean-field regime, since it formally diverges as $N\rightarrow\infty$. The author \cite{Rosenzweig2021ne} later extended the range of scalings for which the supercritical mean-field limit is valid to $\theta \in (1-\frac{2}{d},1)$ and gave arguments for the sharpness of this range in all dimensions. We also mention that the combined mean-field and quasineutral regime has been studied in \cite{GpI2018, GpI2020sing} without the monokinetic assumption, where the expected limiting equation is the so-called \emph{kinetic Euler equation}, but only for very $\vep$ vanishing vary slowly as $N\rightarrow\infty$.

\medskip
Transitioning from classical to quantum, it is well-known that Newton's second law of motion describes the effective dynamics of an interacting system of quantum particles in the semiclassical regime $0<\hbar \ll 1$. The validity of this classical limit is known to hold at the microscopic level of the Schr\"odinger/von Neumann equation and the macroscopic level of the Hartree equation for certain potentials $V$, including Coulomb \cite{LP1993,MM1993, GIMS1998, APP2011, FLP2012, AMN2013a, AMN2013b, BPSS2016, Lafleche2019, Saffirio2020a, Saffirio2020b, LS2020}. See the next subsection for the aforementioned equations. The Hartree equation, in turn, is the mean-field limit of the Schr\"odinger/von Neumann equation. In the case of bosons of interest to this article, the Coulomb case, sometimes called \emph{Schr\"odinger-Poisson}, has been treated with an explicit rate of convergence \cite{EY2001, RS2009, Pickl2011a}, following earlier work for more regular potentials (e.g. bounded) \cite{Spohn1980,BGM2000}. Many authors have contributed to the mean-field theory when $V$ is not necessarily Coulomb, and we refer to the surveys \cite{Golse2016ln, Rougerie2020ems, Napiorkowski2021sur} and references therein for further information. Several works \cite{NS1981,Spohn1981,GMP2003,FGS2007, PP2009, GMP2016,GP2017,GPP2018,GP2019,CLS2021} have investigated the combined semiclassical and mean-field regimes for general potentials $V$, as the $\hbar$-uniformity of the quantum $N$-body mean-field convergence rate implies mean-field convergence for the $N$-body classical dynamics. We remark that for fixed $\hbar$, the state of knowledge in terms of allowable potential singularities is much better in the quantum, as opposed to classical, setting.

Recently, Golse and Paul \cite{GP2021} considered the Schr\"odinger/von Neumann equation in the combined semiclassical and mean-field regimes. Under some technical conditions, they proved that in the limit as $\hbar+\frac{1}{N}\rightarrow 0^+$, the density and current of the first marginal of the solution to the von Neumann equation with approximately monokinetic initial data converge to a solution of the pressureless Euler-Poisson equation. One may view the result of Golse-Paul as the quantum counterpart to the aforementioned result of Duerinckx and Serfaty \cite[Appendix]{Serfaty2020} on the derivation of the Euler-Poisson equation from Newton's second law in the monokinetic regime.


\subsection{Informal statement of results}\label{ssec:introres}
The preceding discussion, in particular the recent works \cite{HkI2021, GP2021}, raises the interesting question of whether one can obtain a nontrivial limit from quantum dynamics in the semiclassical regime where the interaction potential has mean-field supercritical scaling. The purpose of this article is to affirmatively answer this question by showing that the limiting dynamics are given the incompressible Euler equation. This provides, to the best of our knowledge, a new connection between ideal hydrodynamics and quantum many-body systems, complementing the body of work on \emph{quantum hydrodynamics} \cite{Madelung1927, GM1997, KMM2019}, in particular the known link between the Schr\"odinger equation and compressible fluids. We give here an informal statement of the results of the article. For a mathematically precise presentation, see \cref{sec:pres}.

For given $\hbar,\vep>0$ and $N\geq 2$, we consider the $N$-body quantum Hamiltonian
\begin{equation}\label{eq:hamdef}
H_{\hbar,\vep,N}  \coloneqq \sum_{j=1}^N - \frac{\hbar^2}{2}\D_{x_j} + \frac{1}{N\vep^2}\sum_{1\leq j<k\leq N} V(X_j-X_k).
\end{equation}
$N\in\N$ is the number of particles, which we think of as being large, while $\hbar$ is the semiclassical parameter and $\vep>0$ is a  parameter which deviates the scaling from the usual mean-field. We think of both $\hbar$ and $\vep$ as being small. $\D_{x_j}$ denotes the Laplacian on $\T^d$ acting on the j-th particle. $V$ is the Coulomb potential on $\T^d$, i.e. the unique solution to $-\D V = \d_0 - 1$. The notation $V(X_j-X_k)$ denotes the multiplication operator with symbol $V(x_j-x_k)$. Note that the interaction is repulsive.

The evolution of a system of $N$ bosons is governed by the \emph{Schr\"odinger equation}
\begin{equation} \label{eq:Sch}
\begin{cases}
i\hbar\p_t\Phi_{\hbar,\vep,N} = H_{\hbar,\vep,N}\Phi_{\hbar,\vep,N} \\
\Phi_{\hbar,\vep,N}|_{t=0} = \Phi_{\hbar,\vep,N}^0
\end{cases}
\qquad (t,\ux_N) \in \R\times (\T^d)^N,
\end{equation}
where $\ux_N \coloneqq (x_1,\ldots,x_N)$. Given a wave function $\Phi_{\hbar,\vep,N}^t$, we pass to the associated density matrix $R_{\hbar,\vep,N}^t\coloneqq \ket*{\Phi_{\hbar,\vep,N}^t}\bra*{\Phi_{\hbar,\vep,N}^t}$.\footnote{We use here and throughout this article Dirac's bra-ket notation: for $f,g,h\in L^2(\T^d)$, the operator $\ket*{f}\bra*{g}:L^2(\T^d)\rightarrow L^2(\T^d)$ is defined by $(\ket*{f}\bra*{g})h = \ip{g}{h}_{L^2}f$. The integral kernel of $\ket*{f}\bra*{g}$ is $f(x)\ol{g(x')}$.\label{fn:Dirac}} An elementary computation shows that $R_{\hbar,\vep,N}$ is a solution to the $N$-body \emph{von Neumann equation}
\begin{equation}\label{eq:VN}
\begin{cases}
i\hbar\p_t R_{\hbar,\vep,N} = \comm{H_{\hbar,\vep,N}}{R_{\hbar,\vep,N}} \\
\displaystyle \\
R_{\hbar,\vep,N}|_{t=0} = R_{\hbar,\vep,N}^0,
\end{cases}
\qquad (t,\ux_N) \in \R\times (\T^d)^N.
\end{equation}
Above, $\comm{\cdot}{\cdot}$ denotes the usual commutator bracket.

When $N$ is very large, we are interested in statistics of the system. To this end, we consider the \emph{k-particle marginals}, alternatively reduced density matrices, $R_{\hbar,\vep,N:k}$ defined by tracing out $N-k$ particles. Focusing on the 1-particle marginal $R_{\hbar,\vep,N:1}$, the restriction to the diagonal of its integral kernel, denoted by $\rho_{\hbar,\vep,N:1}$, defines a probability density function, which we refer to as the \emph{density} of $R_{\hbar,\vep,N:1}$. Using equation \eqref{eq:VN}, one can show that $\rho_{\hbar,\vep,N:1}$ satisfies the conservation law
\begin{equation}\label{eq:rho}
\begin{cases}
\p_t\rho_{\hbar,\vep,N:1} + \div J_{\hbar,\vep,N:1}= 0 \\
\rho_{\hbar,\vep,N:1}|_{t=0} = \rho_{\hbar,\vep,N:1}^0,
\end{cases}
\end{equation}
where $\rho_{\hbar,\vep,N:1}^0$ is the density of $R_{\hbar,\vep,N:1}^0$ and $J_{\hbar,\vep,N:1}$ is the \emph{current} of $R_{\hbar,\vep,N:1}$ (see \eqref{eq:currdef} below for definition), which is a $d$-dimensional vector of signed Radon measures. As stated precisely in \cref{thm:main}, we assume the initial datum $R_{\hbar,\vep,N}^0$ to be nearly monokinetic and the initial density $\rho_{\hbar,\vep,N}^0$ to be nearly uniform.

If $\vep+\hbar+\frac{1}{N}\rightarrow 0^+$ and
\begin{equation} \label{eq:eprate}
\frac{1+(\log N)\indic_{d=2}}{\vep^2 N^{2/d}} \xrightarrow[\vep+\frac{1}{N}\rightarrow 0^+]{} 0,
\end{equation}
then we show (see \cref{thm:main}) that the density $\rho_{\hbar,\vep,N:1}^t$ converges to $1$ and the current $J_{\hbar,\vep,N:1}^t$ converges to $u^t$, where $u$ is the solution of the incompressible Euler equation \eqref{eq:Eul}, in the sense of measures. Above, $\indic_{(\cdot)}$ denotes the indicator function for the argument $(\cdot)$. Thus, we give a rigorous derivation of the incompressible Euler equation from quantum many-body dynamics. Note that our assumed scaling allows for $\hbar$ and $1/N$ to vanish at independent rates. Additionally, observe that we can interpret our scaling of the interaction potential as a supercritical mean-field limit by writing $\vep^2 N = N^{\theta}$ for $\theta \in (1-\frac{2}{d},1)$. This is exactly the same range obtained by the author \cite{Rosenzweig2021ne} for the derivation of incompressible Euler from Newton's second law. Finally, we remark that our proof is, in fact, quantitative: we give an estimate (see \cref{rem:quant}) for the differences $\rho_{\hbar,\vep,N:1}^t-1$ and $J_{\hbar,\vep,N:1}^t-u^t$ in a negative-order fractional Sobolev spaces, which holds for arbitrary $\hbar,\vep,N$. 

Let us mention that we are not the first to attempt to derive Euler's equations from quantum dynamics. Most notably, Nachtergaele and Yau \cite{NY2002, NY2003} showed, under some technical conditions on the dynamics, that the Heisenberg evolution of the energy, momentum, and particle densities for fermions with short-range interactions  converges to the compressible Euler equation in the hydrodynamic limit. The relationship of our work to \cite{NY2002, NY2003} is unclear given that we have different starting points (bosons vs. fermions), scaling limits (incompressible vs. compressible), and techniques (modulated energy vs. relative entropy).

That we are able to derive the \emph{dynamics} of the incompressible Euler equation from that of the $N$-body Schr\"odinger/von Neumann equation raises the interesting question of what other structure of Euler's equation can be similarly derived. For instance, the incompressible Euler equation is well-known to be a \emph{Hamiltonian} flow on a Poisson manifold \cite{Arnold1966, EM1970, Olver1982, Khesin2009}. The Schr\"odinger and von Neumann equations are similarly known to possess Hamiltonian structures. Given the recent work  of the author together with Mendelson et al. \cite{MNPRS2020adv} on the rigorous derivation of the Hamiltonian structure of the nonlinear Schr\"odinger equation from the N-body Schr\"odinger equation \eqref{eq:Sch} in a different asymptotic regime, it is natural to ask if the Hamiltonian and other geometric structure of the Euler equation can be derived from that of \eqref{eq:Sch} in the scaling regime of the present article. We leave this question for future investigation.

\medskip
The same methods that permit our derivation of the incompressible Euler equation \eqref{eq:Eul} from the von Neumann equation \eqref{eq:VN} in the combined classical and supercritical mean-field limits (alternatively, combined classical, quasineutral, and mean-field limits), also permit us to derive the incompressible Euler equation from the \emph{nonlinear Hartree equation} (written in the formalism of density matrices)
\begin{equation}\label{eq:Har}
\begin{cases}
i\hbar\p_t R_{\hbar,\vep} = \comm{-\frac{\hbar^2}{2}\D + \frac{1}{\vep^2}V\ast\rho_{\hbar,\vep}}{R_{\hbar,\vep}} \\
\rho_{\hbar,\vep} = R_{\hbar,\vep}|_{x=y} \\
R_{\hbar,\vep}|_{t=0} = R_{\hbar,\vep}^0,
\end{cases}
\qquad (t,x) \in\R\times\T^d,
\end{equation}
which is the mean-field limit of \eqref{eq:VN}. For a pure state, equation \eqref{eq:Har} reduces to the \emph{Schr\"odinger-Poisson equation}
\begin{equation}\label{eq:SP}
\begin{cases}
i\hbar\p_t\phi_{\hbar,\vep} = -\frac{\hbar^2}{2}\D\phi_{\hbar,\vep} + \frac{1}{\vep^2}(V\ast|\phi_{\hbar,\vep}|^2)\phi_{\hbar,\vep} \\
\phi_{\hbar,\vep}|_{t=0} = \phi_{\hbar,\vep}^0.
\end{cases}
\end{equation}
Our result (see \cref{thm:mainH}) proves convergence with an explicit rate, as $\hbar+\vep\rightarrow 0^+$, of the density $\rho_{\hbar,\vep}$ and current $J_{\hbar,\vep}$ of \eqref{eq:Har} to the uniform distribution and the solution $u$ of the incompressible Euler equation, respectively, provided that convergence holds in a suitable sense at initial time.

Previously, Puel \cite[Theorem 1.2]{Puel2002} showed for the problem on $\R^d$ that if the initial current $J_{\hbar,\vep}^0$ of $\phi_{\hbar,\vep}^0$ is divergence-free, square-integrable, and approximately monokinetic, then in the combined classical and quasineutral limits, a subsequence of the currents $J_{\hbar,\vep}$ of $\phi_{\hbar,\vep}$ converges in a certain weak topology to a dissipative solution of the incompressible Euler equation \eqref{eq:Eul} with initial datum $J_{\hbar,\vep}^0$. To our knowledge, such dissipative solutions are in general unknown to be unique, but they satisfy a weak-strong uniqueness principle so that uniqueness holds if $J_{\hbar,\vep}^0$ is sufficiently regular (e.g. $C^{1,\al}$) \cite{Lions1996}. We note that Puel's Theorem 1.2 does not give a rate of convergence, nor does the convergence hold in any norm sense, in contrast to our \cref{thm:mainH} stated below.  Puel also proved an estimate (see Proposition 3.1 in the cited work) for the difference between the current $J_{\hbar,\vep}^t$ and a smooth Euler solution $u^t$. Although this result as stated does not give norm convergence and is only for pure states in $\R^d$, for $d=3$, it is more closely comparable to our \cref{thm:mainH}.

\subsection{Comments on the proof}\label{ssec:introcomm}
Inspired by the aforementioned work of Golse and Paul \cite{GP2021} on the derivation of the pressureless Euler-Poisson equation from the many-body equation \eqref{eq:VN} in the combined classical and mean-field limits, the proof of \cref{thm:main} is based on a Gronwall-type relation for the \emph{quantum modulated energy}
\begin{multline}\label{eq:qme}
\G_{\hbar,\vep,N}(R_{\hbar,\vep,N}^t,u^t) \coloneqq \frac{1}{N}\sum_{j=1}^N\Tr_{\H_N}\paren*{\sqrt{R_{\hbar,\vep,N}^t}|\hbar D_{x_j}-u^t(X_j)|^2 \sqrt{R_{\hbar,\vep,N}^t}} \\
+ \frac{1}{\vep^2}\int_{(\T^d)^N}\Fr_N(\ux_N,1+\vep^2\Uu^t)d\rho_{\hbar,\vep,N}^t(\ux_N),
\end{multline}
where
\begin{equation}
\Fr_N(\ux_N,1+\vep^2\Uu^t) \coloneqq \int_{(\T^d)^2\setminus\triangle}V(x-y)d\paren*{\frac{1}{N}\sum_{j=1}^N \d_{x_j} - 1-\vep^2\Uu^t}^{\otimes 2}(x,y).
\end{equation}
Above, $\rho_{\hbar,\vep,N}^t$ is the density associated to $R_{\hbar,\vep,N}^t$, $D_{x_j} \coloneqq -i\nabla_{x_j}$, and $u^t(X_j)$ is the vector-valued multiplication operator with symbol $u^t(x_j)$. $\triangle$ denotes the diagonal of $(\T^d)^2$. The trace is taken over the N-particle Hilbert space $\H_N=L^2((\T^d)^N)$. $\Uu^t$ is a corrector about which we say more below. We abuse notation here and throughout the article by using the same symbol for both a measure and its density.

The first term in the right-hand side of \eqref{eq:qme} is the kinetic energy portion and is as in the Golse-Paul quantum modulated energy (see beginning of Section 3 in their work), except now $u$ solves the incompressible Euler equation rather than the pressureless Euler-Poisson equation. Note that by using the bosonic symmetry of $R_{\hbar,\vep,N}^t$, we have the identity
\begin{multline}
\frac{1}{N}\sum_{j=1}^N\Tr_{\H_N}\paren*{\sqrt{R_{\hbar,\vep,N}^t}|\hbar D_{x_j}-u^t(X_j)|^2 \sqrt{R_{\hbar,\vep,N}^t}}\\
 = \Tr_{\H}\paren*{\sqrt{R_{\hbar,\vep,N:1}^t}|\hbar D-u^t(X)|^2 \sqrt{R_{\hbar,\vep,N:1}^t}},
\end{multline}
where now the trace is taken over the 1-particle Hilbert space $\H=L^2(\T^d)$ and the density matrices are the 1-particle marginals.

The second term in the right-hand side of \eqref{eq:qme}, which is the potential energy portion, corresponds to the expectation of the classical modulated energy from Serfaty's work \cite{Serfaty2020} with inputs $\ux_N$ and $1+\vep^2\Uu^t$, where the $\ux_N$ are regarded as randomly distributed with law $\rho_{\hbar,\vep,N}^t$. The quantity $\Fr_N(\ux_N,1+\vep^2\Uu^t)$ can be understood as a renormalization of the $\dot{H}^{-1}$ norm of $\frac{1}{N}\sum_{j=1}^N\d_{x_j}-1-\vep^2\Uu^t$. Note that in contrast to the classical setting, the positions $\ux_N$ are not time-dependent, reflecting the fact that there is not a precise notion of point-particle dynamics in the quantum setting.\footnote{Bohmian mechanics \cite{deBroglie1927,Bohm1952i, Bohm1952ii}, alternatively pilot-wave theory, attempts to supplement the Schr\"odinger wave function with a notion of particle trajectories. We refer the interested reader to the book \cite{DT2009} for further discussion on theory and to the survey \cite{BO2020} for some recent experimental developments.}  Observe that by symmetry with respect to permutation of particle labels, we have the identity
\begin{multline}\label{rem:pMEsym}
\int_{(\T^d)^N}\Fr_N(\ux_N,1+\vep^2\Uu^t)d\rho_{\hbar,\vep,N}^t(\ux_N)= \frac{(N-1)}{N}\int_{(\T^d)^2}V(x_1-x_2)d\rho_{\hbar,\vep,N:2}^t(x_1,x_2) \\
-2\int_{\T^d} (-\D)^{-1}(1+\vep^2\Uu^t)(x_1)d\rho_{\hbar,\vep,N:1}^t(x_1) + \int_{(\T^d)^2}V(x-y)d\paren*{1+\vep^2\Uu^t}^{\otimes 2}(x,y).
\end{multline}

Compared to the Golse-Paul functional, a key difference in our choice of \eqref{eq:qme} is the presence of the corrector $\vep^2\Uu^t \coloneqq \vep^2\p_{x_\al}(u^t)^{\be}\p_{x_\be}(u^t)^\al$, with implicit summation over $\al,\be$, the choice of which is motivated by the aforementioned work of Han-Kwan and Iacobelli \cite{HkI2021} and the author \cite{Rosenzweig2021ne}. To understand the intuition behind the corrector, recall that the pressure $p$ is obtained from the velocity field $u$ through the Poisson equation $-\D p = \Uu$, which follows by taking the divergence of both sides of \eqref{eq:Eul} and using that $u$ is divergence-free. As will be clear during the proof of \cref{prop:MEtd} below, the corrector is necessary to cancel out a contribution of the pressure term when we compute the variation of the quantum modulated energy with respect to time (see \eqref{eq:presscan}). The technical novelty of our work consists, in part, of showing that such a corrector may be incorporated into a suitable quantum modulated energy.

In some sense, the original functional used by Golse and Paul is the quantum analogue of the modulated energy used by Duerinckx and Serfaty \cite[Appendix]{Serfaty2020} to derive the pressureless Euler-Poisson equation from the Newton's second law with mean-field scaling in the monokinetic regime. Similarly, one may view our functional \eqref{eq:qme} as the quantum analogue of the modulated energy used by Han-Kwan and Iacobelli \cite{HkI2021} and the author \cite{Rosenzweig2021ne} to derive the incompressible Euler equation from Newton's second law with supercritical mean-field scaling in the monokinetic regime.

\begin{remark}\label{rem:MElb}
As is by now well-known (see \cite[Proposition 3.3]{Serfaty2020}), the functional $\Fr_N(\ux_N,\mu)$, for any bounded density $\mu$, is not necessarily a nonnegative quantity. But there is a constant $C_d>0$, depending only on the dimension $d$, such that
\begin{equation}\label{eq:MElb}
\Fr_N(\ux_N,\mu) + \frac{C_d(1+\|\mu\|_{L^\infty})(1+(\log N)\indic_{d=2})}{N^{2/d}} \geq 0.
\end{equation}
\end{remark}

Differentiating the functional $\G_{\hbar,\vep,N}(R_{\hbar,\vep,N}^t,u^t)$ with respect to time (see \cref{prop:MEtd}), we find that to close our Gronwall argument, we need to estimate, among other terms, the expression
\begin{equation}\label{eq:introcomm}
\frac{1}{\vep^2}\left|\int_{(\T^d)^2\setminus\triangle} \paren*{u^t(x)-u^t(y)}\cdot\nabla V(x-y)d\paren*{\frac{1}{N}\sum_{j=1}^N\d_{x_j}-1-\vep^2\Uu^t}^{\otimes 2}(x,y)\right|
\end{equation}
for a given configuration $\ux_N = (x_1,\ldots,x_N)\in (\T^d)^N$. Such expressions were originally understood in terms of a stress-energy tensor structure \cite{Duerinckx2016,Serfaty2020}. Subsequently, they were reinterpreted in terms of a \emph{commutator} that has been \emph{renormalized} through the excision of the diagonal in order to allow for the singularity of the Dirac masses \cite{Rosenzweig2020spv, NRS2021}. We can estimate (see \cref{prop:comm}) the expression \eqref{eq:introcomm} in terms of $\frac{1}{\vep^2}\Fr_N(\ux_N,1+\vep^2\Uu^t)$ up to an error of size $\frac{(1+(\log N)\indic_{d=2})}{N^{2/d}\vep^2}$. It is at this step in the analysis where we obtain the scaling restriction \eqref{eq:eprate}.

We remark that $\frac{(1+(\log N)\indic_{d=2})}{N^{2/d}\vep^2}$ also appears in the lower bound \eqref{eq:MElb} after multiplying both sides by $\vep^{-2}$. It is known in the static case on $\R^d$ with a confining potential (e.g., see \cite[Introduction]{RS2016}) that $\frac{1+(\log N)\indic_{d=2}}{N^{2/d}}$ gives the optimal scaling for the functional $\Fr_N(\cdot,\cdot)$. If this is also the case in our dynamical setting, then we expect our scaling relation \eqref{eq:eprate} to be sharp. Writing $\vep^2 N = N^\theta$, we see that we can allow for $1-\frac{2}{d}<\theta < 1$, just as in the author's previous derivation of Euler from Newton \cite{Rosenzweig2021ne}. We leave the very interesting question of the effective limiting dynamics at and below the threshold $\theta = 1-\frac{2}{d}$ for future investigation.

\subsection{Organization of article}\label{ssec:introorg}
Let us briefly comment on the organization of the remaining body of the article. In \cref{sec:pres}, we introduce the basic notation of the paper (\cref{ssec:presnot}), review preliminary facts about the Cauchy problem for the von Neumann and Hartree equations (\cref{ssec:presset}), and precisely state our main results \Cref{thm:main,thm:mainH} along with extensive commentary (\cref{ssec:presmt}). \cref{sec:Mpf} is the main part of the article and devoted to the proofs of \Cref{thm:main,thm:mainH} along with the details behind the convergence assertions of \cref{rem:quant}. The section consists of several subsections: \Cref{ssec:Mpftd,ssec:Mpffi,ssec:Mpfgron} prove the main quantum modulated energy estimate \eqref{eq:thmmain}, \cref{ssec:Mpfmain} completes the proof of \cref{thm:main} and elaborates on \cref{rem:quant}, and \cref{ssec:MpfH} completes the proof of \cref{thm:mainH}.

\subsection{Acknowledgments}\label{ssec:introack}
The author thanks Thomas Chen, Fran\c{c}ois Golse, and Sylvia Serfaty for their correspondence related to this project. The author also gratefully acknowledge the hospitality of the Institute for Computational and Experimental Research in Mathematics (ICERM) where the manuscript for this project was completed during the ``Hamiltonian Methods in Dispersive and Wave Evolution Equations'' semester program.

\section{Presentation of results}\label{sec:pres}
In this section, we give a mathematically precise statement of the results of this paper, which we informally described in \cref{ssec:introres}.

\subsection{Basic notation}\label{ssec:presnot}
We introduce here the basic notation of the article.

Given nonnegative quantities $A$ and $B$, we write $A\lesssim B$ if there exists a constant $C>0$, independent of $A$ and $B$, such that $A\leq CB$. If $A \lesssim B$ and $B\lesssim A$, we write $A\sim B$. To emphasize the dependence of the constant $C$ on some parameter $p$, we sometimes write $A\lesssim_p B$ or $A\sim_p B$. Similarly, we use $C$ for a generic constant, with subscripts to emphasize its dependence on parameters. We denote the natural numbers excluding zero by $\N$ and including zero by $\N_0$. We denote the positive real numbers by $\R_+$. Given $N\in\N$ and points $x_{1},\ldots,x_{N}$ in some set $X$, we will write $\ux_N$ to denote the $N$-tuple $(x_{1},\ldots,x_{N})$.

$\M(\T^d)$ denotes the space of complex-valued Borel measures on $\T^d$, and $\P(\T^d)$ denotes the subspace of probability measures (i.e. elements $\mu\in\M(\T^d)$ with $\mu\geq 0$ and $\mu(\T^d)=1$). For $\mu$ absolutely continuous with respect to Lebesgue measure on $\T^d$, we abuse notation by writing $\mu$ for both the measure and its density function. We denote the Banach space of complex-valued continuous functions on $\T^d$ by $C(\T^d)$ equipped with the uniform norm $\|\cdot\|_{\infty}$. More generally, we denote the Banach space of $k$-times continuously differentiable functions with bounded derivatives up to order $k$ by $C^k(\T^d)$ equipped with the natural norm. $C^{k,\al}$ denotes the space of functions whose $k$-th order partial derivatives are $\al$-H\"older continuous, for $0<\al\leq 1$. We set $C^\infty \coloneqq \bigcap_{k=1}^\infty C^k$, and we use the subscript $c$ to indicate functions with compact support. $L^p(\Omega)$ denotes the usual Banach space of $p$-integrable functions with norm $\|\cdot\|_{L^p}$. The space of distributions on $\T^d$, equivalently periodic distributions on $\R^d$, is denote by $\Dc'(\T^d)$.

\subsection{Quantum setup}\label{ssec:presset}
We first expose the dynamics at the microscopic level of finite $\hbar,\vep,N$. This subsection follows, in part, the presentation of \cite[Section 2]{GP2021}.

We recall from the introduction that $V$ denotes the element of $\Dc'(\T^d)$
\begin{equation}
\label{eq:VFS}
\frac{1}{(2\pi)^2}\sum_{{k\in\Z^d} : {k\neq 0}} \frac{e^{2\pi ik\cdot x}}{|k|^2},
\end{equation}
which is, in fact, an element of $C^\infty(\T^d\setminus\{0\})$ satisfying
\begin{equation}
-\D V = \d_{0} - 1.
\end{equation}
Here and throughout the article, it is convenient to identify $\T^d$ with $[-\frac{1}{2},\frac{1}{2}]^d$ with periodic boundary conditions. Let $V_{\R^d}$ denote the Coulomb potential on $\R^d$:
\begin{equation}
\label{eq:VRd}
V_{\R^d}(x) \coloneqq  \begin{cases}
-\frac{1}{2\pi}\log|x|, & {d=2} \\
c_d|x|^{2-d}, & {d\geq 3},
\end{cases}
\end{equation}
where the normalizing constant $c_d = \frac{1}{d(d-2)|B(0,1)|}$. As is well-known, there exists a function $V_{loc} \in C^\infty(\ol{B(0,1/4)})$ such that 
\begin{equation}
\label{eq:Vasy}
V(x) = V_{\R^d}(x) + V_{loc}(x) \qquad \forall x\in\ol{B(0,1/4)}.
\end{equation}

Let $\H\coloneqq L^2(\T^d)$ and $\H_N \coloneqq L^2((\T^d)^N) \simeq \H^{\otimes N}$ be the $1$-particle and $N$-particle Hilbert spaces, respectively. By Sobolev embedding and the Kato-Rellich theorem \cite[Theorem X.12]{RSII}, the quantum Hamiltonian $H_{\hbar,\vep,N}$ defined in \eqref{eq:VN} is a self-adjoint operator on domain $H^2((\T^d)^N)$. Stone's theorem \cite[Theorem VII.8]{RSI} tells us that $H_{\hbar,\vep,N}$ generates a one-parameter unitary group $\{e^{it H_{\hbar,\vep,N}}\}_{t\in\R}$ on the Hilbert space $\H_N$.

Since we ultimately want to consider mixed states, let us start directly with the Cauchy problem for the von Neumann equation \eqref{eq:VN}. To this end, we denote the usual Schatten spaces by $\L^p(\H_N)$, for $1\leq p\leq \infty$, with associated norm $\|\cdot\|_p$. $\L^1$ are the trace-class operators, $\L^2$ Hilbert-Schmidt, and $\L^\infty$ compact. The space of bounded operators on $\H_N$ is denoted by $\L(\H_N)$. The subset of $\L^1$ consisting of density matrices (i.e. nonnegative, self-adjoint operators with unit trace) on a given Hilbert space $\mathscr{H}$ is denoted by $\Dc(\mathscr{H})$. In the case $\mathscr{H}=\H_N$, we use the subscript $s$ to denote those operators $R_N$ which have bosonic symmetry:
\begin{equation}
U_{\sigma}R_NU_{\sigma}^{*}=R_N \qquad \forall \sigma\in\Ss_N,
\end{equation}
where
\begin{equation}
U_{\sigma}\Phi(x_1,\ldots,x_N) \coloneqq \Phi(x_{\sigma^{-1}},\ldots,x_{\sigma^{-1}(N)}).
\end{equation}
Given initial datum $R_{\hbar,\vep,N}^{0} \in \Dc(\H_N)$, the unique solution to equation \eqref{eq:VN} in the class $C(\R; \Dc(\H_N))$ is given by
\begin{equation}
R_{\hbar,\vep,N}^t \coloneqq e^{-itH_{\hbar,\vep,N}/\hbar}R_{\hbar,\vep,N}^0 e^{itH_{\hbar,\vep,N}/\hbar}.
\end{equation}
Moreover, if $R_{\hbar,\vep,N}^0\in \Dc_s(\H_N)$, then $R_{\hbar,\vep,N}^t\in \Dc_s(\H_N)$ for all $t\in\R$.

Any density matrix $R_N$ on $\H_N$ has an integral kernel, which we also denote by $R_N$ with an abuse of notation. Since trace-class operators are a fortiori Hilbert-Schmidt, this kernel belongs to $L^2((\T^d)^N \times (\T^d)^N)$. In fact, the map $x\mapsto R_N(\cdot,\cdot+x)$ belongs to $C((\T^d)^N; L^1((\T^d)^N))$ \cite[footnote 1, pp. 61-62]{GP2017}. Consequently, the kernel has a well-defined restriction to the diagonal denoted by $\rho_N$ and called the \emph{density} of $R_N$. Since $\rho_N$ is nonnegative pointwise a.e., belongs to $L^1((\T^d)^N)$, and has unit integral, $\rho_N$ is, indeed, a probability density function.

In addition to a density, a density matrix $R_N$ on $\H_N$ defines a {current}. Indeed, let $D_{\ux_N}\coloneqq -i\nabla_{\ux_N}$, where $\nabla_{\ux_N} = (\nabla_{x_1},\ldots,\nabla_{x_N})$, and suppose that
\begin{equation}
\Tr_{\H_N}\paren*{\sqrt{R_N}|D_{\ux_N}|^2\sqrt{R_N}}<\infty.
\end{equation}
Then for each $k\in\{1,\ldots,N\}$ and $\al\in\{1,\ldots,d\}$, we can use the Riesz representation theorem \cite[Theorem IV.17]{RSI} to define a signed Radon measure $J_{k,N}^{\al}$ on $(\T^d)^N$ by the formula
\begin{equation}\label{eq:currdef}
\int_{(\T^d)^N}\phi(\ux_N)dJ_{k,N}^\al(\ux_N) \coloneqq \Tr_{\H_N}\paren*{ \paren*{\phi(\ux_N)\vee \frac{\hbar}{2}D_{x_k^\al}}R_N}, \qquad \forall \phi\in C((\T^d)^N).
\end{equation}
Here, $\vee$ denotes the anti-commutator of two operators defined by
\begin{equation}
A \vee B \coloneqq AB+BA.
\end{equation}
The vector $J_N \coloneqq (J_{1,N},\ldots,J_{N,N})$, where $J_{k,N}\coloneqq (J_{k,N}^1,\ldots,J_{k,N}^d)$, is a signed Radon-measure-valued vector field called the \emph{current} of $R_N$.

Additionally, if $R_N\in \Dc_s(\H_N)$, then we define the k-th marginal $R_{N:k}$ to be the unique trace-class operator on the k-particle Hilbert space $\H_k$ with the property
\begin{equation}
\Tr_{\H_k}\paren*{R_{N:k}A_k} = \Tr_{\H_N}\paren*{R_{N}\paren*{A_k \otimes \Ib_{\H_{N-k}}}} \qquad \forall A_k \in \L(\H_k),
\end{equation}
where $\Ib_{\H_{N-k}}$ denotes the identity operator on the $(N-k)$-particle Hilbert space $\H_{N-k}$.

Suppose that $R_{\hbar,\vep,N}^{0} \in \Dc_s(\H_N)$ is an N-particle bosonic density matrix satisfying the regularity condition
\begin{equation}\label{eq:iddmreg}
\Tr_{\H_N}\paren*{\sqrt{R_{\hbar,\vep,N}^0}\paren*{\Ib_{\H_N}-\frac{\hbar^2}{2}\sum_{j=1}^N \D_{x_j}}^2 \sqrt{R_{\hbar,\vep,N}^0}} <\infty.
\end{equation}
If $R_{\hbar,\vep,N}$ is the solution to \eqref{eq:VN} with initial datum $R_{\hbar,\vep,N}^{0}$, then one can show (see \cite[pp. 11-12]{GP2021}) that $\|H_{\hbar,\vep,N}\sqrt{R_{\hbar,\vep,N}^t}\|_2$ is controlled by the left-hand side of \eqref{eq:iddmreg}, hence finite, for all $t\in\R$. In particular, $R_{\hbar,\vep,N}^t$ has finite, conserved energy $\Tr_{\H_N}(H_{\hbar,\vep,N}R_{\hbar,\vep,N}^t)$. Additionally, for such $R_{\hbar,\vep,N}^t$, the quantum modulated energy $\G_{\hbar,\vep,N}(R_{\hbar,\vep,N}^t,u^t)$ is well-defined. Indeed, the kinetic term is finite by conservation of mass and energy together with $u^t\in L^\infty$. To see that the potential term is finite, we use the identity \cref{rem:pMEsym}. Note that
\begin{equation}
\frac{(N-1)}{N}\left|\int_{(\T^d)^2}V(x_1-x_2)d\rho_{\hbar,\vep,N:2}^t(x_1,x_2)\right| \lesssim C_d+ \frac{1}{N}\Tr_{\H_N}\paren*{H_{\hbar,\vep,N}R_{\hbar,\vep,N}^t}.
\end{equation}
Also, observe from \eqref{eq:Vasy} that $\|(-\D)^{-1}(1+\vep^2\Uu^t)\|_{L^\infty} \lesssim_d \vep^2\|\nabla u^t\|_{L^\infty}^2$, which is finite by assumption. Since $\rho_{\hbar,\vep,N:1}^t$ is a probability density, it follows that
\begin{multline}
\left|\int_{\T^d} (-\D)^{-1}(1+\vep^2\Uu^t)(x_1)d\rho_{\hbar,\vep,N:1}^t(x_1)\right| + \left|\int_{(\T^d)^2}V(x-y)d\paren*{1+\vep^2\Uu^t}^{\otimes 2}(x,y)\right|\\
 \lesssim_d \vep^2\|\nabla u^t\|_{L^\infty}^2\paren*{1+\vep^2\|\nabla u^t\|_{L^\infty}^2}.
\end{multline}

\medskip
Lastly, we discuss the 1-particle mean-field dynamics of the nonlinear Hartree equation \eqref{eq:Har}. If the initial datum $R_{\hbar,\vep}^0$ satisfies the regularity condition
\begin{equation}\label{eq:iddmregH}
\Tr_{\H}\paren*{\sqrt{R_{\hbar,\vep}^0}|D|^4\sqrt{R_{\hbar,\vep}^0}} < \infty,
\end{equation}
then it is classical (e.g., see \cite{BDpF1976}) that the Cauchy problem \eqref{eq:Har} admits a unique, global solution $R_{\hbar,\vep}$ in $C(\R; \Dc(\H))$. Moreover, if we introduce the 1-particle time-dependent Hamiltonian (i.e. the mean-field Hamiltonian)
\begin{equation}\label{eq:mfham}
H_{\hbar,\vep}^t \coloneqq -\frac{\hbar}{2}\D + \frac{1}{\vep^2}(V\ast\rho_{\hbar,\vep}^t)(X),
\end{equation}
where $\rho_{\hbar,\vep}^t$ is the density of $R_{\hbar,\vep}^t$, then $\|H_{\hbar,\vep}^t\sqrt{R_{\hbar,\vep}^t}\|_2$ is finite for all $t$ and the energy $\Tr_{\H}(H_{\hbar,\vep}^tR_{\hbar,\vep}^t)$ is conserved.

\subsection{Main theorems}\label{ssec:presmt}
With the preliminaries out of the way, we are now ready to state our main results. Our first theorem (cf. \cite[Theorem 2.2]{GP2021}, \cite[Theorem 1.1]{HkI2021}, \cite[Theorem 1.1]{Rosenzweig2021ne}) shows that as $N\rightarrow \infty$ and $\hbar\rightarrow 0^+$ with independent rates, while $\vep\rightarrow 0^+$ respecting \eqref{eq:eprate}, the k-particle density converges to the uniform distribution while the k-particle current converges to a vector field completely determined by a solution of the incompressible Euler equation on an interval for which the equation admits a classical solution, provided that the initial quantum modulated energy tends to zero in the same limit.

\begin{thm}\label{thm:main}
Let $d\in\{2,3\}$, $0<\al\leq 1$. Suppose that $T>0$ is such that $u\in L^\infty([0,T]; C^{1,\al}(\T^d))$ is a solution to the incompressible Euler equation \eqref{eq:Eul}. Given $N\geq 2$ and $\hbar,\vep>0$, let $R_{\hbar,\vep,N}^{0} \in \Dc_s(\H_N)$ satisfy the regularity condition \eqref{eq:iddmreg}. Let $R_{\hbar,\vep,N}$ denote the solution of the von Neumann \eqref{eq:VN} in $C(\R; \Dc_s(\H_N))$ with initial datum $R_{\hbar,\vep,N}^0$. For $1\leq k\leq N$, let $R_{\hbar,\vep,N:k}$ denote the k-particle marginal of $R_{\hbar,\vep,N}$, and let $\rho_{\hbar,\vep,N:k}$ and $J_{\hbar,\vep,N:k}$ respectively denote the density and current of $R_{\hbar,\vep,N:k}$.

There exist constants $C_d, C_{d,\al}>0$ depending only on $d$ and $d,\al$, respectively, such that for every $0\leq t\leq T$,
\begin{multline}\label{eq:thmmain}
\left|\G_{\hbar,\vep,N}(R_{\hbar,\vep,N}^t,u^t)\right| \leq e^{C_d\int_0^t(1+\|\nabla u^\tau\|_{L^\infty})d\tau} \Bigg(\left|\G_{\hbar,\vep,N}(R_{\hbar,\vep,N}^0, u^0)\right| \\
+ \frac{C_d(1+(\log N)\indic_{d=2})}{\vep^2 N^{2/d}}\int_0^t (1+\|\nabla u^\tau\|_{L^\infty})(1+\vep^2\|\nabla u^\tau\|_{L^\infty}^2) d\tau\\
 + C_{d,\al}\int_0^t \paren*{\vep^2 \|u^\tau\|_{C^{1,\al}}^3+N^{-1/d}}\|u^\tau\|_{C^{1,\al}}^3 d\tau\Bigg).
\end{multline}
Consequently, if $\{R_{\hbar,\vep,N}^0\}_{\hbar,\vep,N}$ is a sequence of initial data such that
\begin{equation}
\G_{\hbar,\vep,N}(R_{\hbar,\vep,N}^0, u^0) \xrightarrow[\frac{1}{N}+\hbar+\vep\rightarrow 0^+]{} 0,
\end{equation}
provided that $\vep=\vep(N)$ is chosen so that the limit \eqref{eq:eprate} holds, then for $1\leq j\leq k, \ 1\leq \al\leq d$,
\begin{equation}\label{eq:thmmaincnv}
\rho_{\hbar,\vep,N:k}^t \xrightharpoonup[\frac{1}{N}+\hbar +\vep\rightarrow 0^+]{} 1 \qquad \text{and} \qquad (J_{\hbar,\vep,N:k}^t)_{j}^{\al} \xrightharpoonup[\frac{1}{N}+\hbar+\vep\rightarrow 0^+]{} (u^t)^{\al}
\end{equation}
in the weak-* topology for the space $\M((\T^d)^k)$ of signed Radon measures on $(\T^d)^k$.
\end{thm}

As commented in the introduction, \cref{thm:main} provides a derivation of the incompressible Euler equation directly from a system of interacting bosons in the combined classical and supercritical mean-field limit, provided the initial data are asymptotically monokinetic. To the best of our knowledge, this role of the incompressible Euler equation as an effective equation for the 1-particle current of a system of bosons has not been previously shown. The fact that we obtain an incompressible fluid description in the limit stands in contrast to the well-known connection between quantum mechanics and compressible fluids through the Madelung transformation \cite{Madelung1927, GM1997}.

Before stating our second result, let us record some remarks on the statement of \cref{thm:main} and the assumptions contained in it.

\begin{remark}\label{rem:ID}
We can construct examples of initial data satisfying the conditions of \cref{thm:main} as follows (cf. \cite[Remark 1.4]{Rosenzweig2021ne}).

Suppose that $u_0$ is a smooth, divergence-free vector field on $\T^d$. Let $\la>0$ be a small parameter, and let $\chi_\la\in C_c^\infty(\R^d)$ be a smooth approximation of the characteristic function $1_{[-\frac{1}{2}+2\la,\frac{1}{2}-2\la]^d}$, such that $\supp(\chi_\la) \subset [-\frac{1}{2}+\la,\frac{1}{2}-\la]^d$, $0\leq\chi_\la\leq 1$, and $\|\nabla^{\otimes k}\chi_\la\|_{L^\infty} \lesssim_{k} \la^{-k}$. For a probability density $\rho_{\hbar,\vep}^0$ on $\T^d$, consider a 1-particle density matrix of the form
\begin{equation}
R_{\la,\hbar,\vep}^0(x,x') \coloneqq c_{\la,\hbar,\vep}\sqrt{\rho_{\hbar,\vep}^0(x)}\sqrt{\rho_{\hbar,\vep}^0(x')}\chi_{\la}(x)\chi_{\la}(x')e^{\frac{i(x-x')\cdot u^0(\frac{x+x'}{2})}{\hbar}}, \qquad x,x'\in (-\frac{1}{2},\frac{1}{2})^d,
\end{equation}
with (smooth) periodic extension in each coordinate to all of $\R^d\times\R^d$. The normalizing constant $c_{\la,\hbar,\vep}$ is defined by
\begin{equation}
c_{\la,\hbar,\vep} \coloneqq \paren*{\int_{(-\frac{1}{2},\frac{1}{2})^d} \rho_{\hbar,\vep}^0(x)\chi_{\la}^2(x)dx}^{-1}.
\end{equation}
Set $R_{\la,\hbar,\vep,N}^0 \coloneqq (R_{\la,\hbar,\vep}^0)^{\otimes N}$. For simplicity, we suppose that $\rho_{\hbar,\vep}^0$ is $C^\infty$ and strictly positive everywhere, so that $R_{\la,\hbar,\vep,N}^0$ satisfies the regularity condition \eqref{eq:iddmreg}.

If we choose $\la = \la(\hbar,\vep)$ and assume that $\hbar\|\nabla(\sqrt{\rho_{\hbar,\vep}^0}\chi_{\la})\|_{L^2}\xrightarrow[\hbar+\vep\rightarrow 0^+]{} 0$, then one can check from the chain rule that
\begin{equation}\label{eq:imekelim}
\lim_{\vep+\hbar\rightarrow 0^+} \Tr_{\H}\paren*{\sqrt{R_{\la,\hbar,\vep}^0}|\hbar D-u^0(X)|^2 \sqrt{R_{\la,\hbar,\vep}^0}} = 0, \\
\end{equation}
which takes care of the kinetic portion of $\G_{\hbar,\vep,N}(R_{\la,\hbar,\vep,N}^0,u^0)$.

For the potential energy term, we use \cref{rem:pMEsym} to write
\begin{align}
\frac{1}{\vep^2}\int_{(\T^d)^N}\Fr_N(\ux_N,1+\vep^2\Uu^0)d\rho_{\la,\hbar,\vep,N}^0(\ux_N) = \frac{(N-1)}{N\vep^2}\int_{(\T^d)^2}V(x_1-x_2)d(c_{\la,\hbar,\vep}\rho_{\hbar,\vep}^0\chi_\la^2)^{\otimes 2}(x_1,x_2) \nn\\
-\frac{2}{\vep^2}\int_{\T^d} (-\D)^{-1}\paren*{1+\vep^2\Uu^0}(x_1)d(c_{\la,\hbar,\vep}\rho_{\hbar,\vep}^0\chi_\la^2)(x_1) + \frac{1}{\vep^2}\int_{(\T^d)^2}V(x-y)d\paren*{1+\vep^2\Uu^0}^{\otimes 2}(x,y) \nn\\
=-\frac{1}{N\vep^2}\|c_{\la,\hbar,\vep}\rho_{\hbar,\vep}^0\chi_\la^2 - 1\|_{\dot{H}^{-1}}^2 + \frac{1}{\vep^2}\left\|c_{\la,\hbar,\vep}\rho_{\hbar,\vep}^0\chi_{\la}^2 - 1 - \vep^2\Uu^0\right\|_{\dot{H}^{-1}}^2.
\end{align}
If $\sup_{\hbar,\vep>0} \|\rho_{\hbar,\vep}^0\|_{L^\infty} < \infty$ (an $L^p$ condition for high enough $p$ would suffice), then \\ $\sup_{\hbar,\vep>0} \|\rho_{\hbar,\vep}^0-1\|_{\dot{H}^{-1}}<\infty$. Note that the $\dot{H}^{-1}$ norm is well-defined here and above since the zero-order Fourier modes of the arguments vanish. Also, we can replace $c_{\la,\hbar,\vep}\chi_{\la}^2$ by $1$ above up to an error vanishing as $\la\rightarrow 0^+$. Assuming that $N\vep^2\xrightarrow[\frac{1}{N}+\vep\rightarrow 0^+]{} \infty$, which is more general than our assumed relation \eqref{eq:eprate}, and that
\begin{equation}
\frac{1}{\vep^2}\left\|\rho_{\hbar,\vep}^0 - 1 - \vep^2\Uu^0\right\|_{\dot{H}^{-1}}^2 \xrightarrow[\hbar+\vep\rightarrow 0^+]{} 0,
\end{equation}
it follows that
\begin{equation}\label{eq:imepelim}
\lim_{\frac{1}{N}+\hbar+\vep\rightarrow 0^+} \frac{1}{\vep^2}\int_{(\T^d)^N}\Fr_N(\ux_N,1+\vep^2\Uu^0)d\rho_{\la,\hbar,\vep,N}^0(\ux_N) = 0.
\end{equation}

Together, the limits \eqref{eq:imekelim} and \eqref{eq:imepelim} show that $\G_{\hbar,\vep,N}(R_{\la,\hbar,\vep,N}^0, u^0)$ vanishes as $\hbar+\vep+\frac{1}{N}\rightarrow 0^+$.
\end{remark}

\begin{remark}\label{rem:quant}
The quantitative convergence of \eqref{eq:thmmain} actually implies convergence of the k-particle density $\rho_{\hbar,\vep,N:k}^t$ and current $J_{\hbar,\vep,N:k}^t$ in a tensor product norm of negative-order fractional Sobolev spaces. We have not included this convergence in the statement of \cref{thm:main} so as to simplify the presentation; but we work out the details in \cref{ssec:Mpfmain} and, in fact, deduce the weak-* convergence of measures from this norm convergence. Note that compared to the convergence assertions in \cite{GP2021}, our quantitative norm convergence is new.
\end{remark}

\begin{remark}\label{rem:reg}
It is well-known that given $C^{1,\al}$ initial data, the incompressible Euler equation \eqref{eq:Eul} has a unique $C^{1,\al}$ solution. For instance, see \cite[Sections 7.1-7.2]{BCD2011}. If $d=2$, then the solution is global; while if $d=3$, then only local existence is known. It is an interesting question whether one can relax the $C^{1,\al}$ assumption. From our commentary in \cref{ssec:introcomm}, we need $\nabla u^t\in L^\infty$. Given the recent work \cite{Rosenzweig2020pvmf} by the author on the point-vortex approximation for 2D Euler with log-Lipschitz velocity, we are hopeful that one can relax the regularity assumption if $d=2$.

The regularity assumption \eqref{eq:iddmreg} on the density matrices is taken from \cite{GP2021} and is a technical condition to ensure that certain computations are justified. As noted in the previous subsection, the condition implies that $\|H_{\hbar,\vep,N}\sqrt{R_{\hbar,\vep,N}^t}\|_2 < \infty$ for all $t\in\R$. We make no claim as to its optimality.
\end{remark}

\begin{remark}\label{rem:Wig}
In addition to the convergence assertion \eqref{eq:thmmaincnv} for the density and current, one can---with more work---prove a convergence result for the Wigner function of $R_{\hbar,\vep,N:k}^t$. For instance, one could show that the Wigner function $W_{\hbar,\vep,N:1}^t$ of $R_{\hbar,\vep,N:1}^t$ converges weakly as $\hbar+\vep+\frac{1}{N}\rightarrow 0^+$, in our imposed scaling regime, to the measure
\begin{equation}
\int_{\T^d\times\R^d}\psi(x,\xi)d\d(\xi-u^t(x))(\xi)dx \coloneqq \int_{\T^d}\psi(x,u^t(x))dx, \qquad \psi \in C(\T^d\times\R^d).
\end{equation}
We refer to \cite[Subsection 4.4]{GP2021} for guidance on how to show this.
\end{remark}

\begin{remark}\label{rem:highd}
We have only considered the physically relevant cases $d=2,3$. For $d>3$, one could, in principle, carry out the same proof and obtain the same result as \cref{thm:main}. Some care might be needed when $d$ is very large since the weakness of Sobolev embedding in higher dimensions means that the potential term in the Hamiltonian $H_{\hbar,\vep,N}$ cannot be treated as a perturbation of the Laplacian to show that the operator is self-adjoint on $H^2((\T^d)^N)$. 
\end{remark}

To state our second result, we introduce the 1-particle quantum modulated energy
\begin{equation}\label{eq:ME1def}
\G_{\hbar,\vep}(R_{\hbar,\vep}^t,u^t)\coloneqq \Tr_{\H}\paren*{\sqrt{R_{\hbar,\vep}^t}|\hbar D-u^t(X)|^2\sqrt{R_{\hbar,\vep}^t}} +\frac{1}{\vep^2} \left\|\rho_{\hbar,\vep}^t-1-\vep^2\Uu^t\right\|_{\dot{H}^{-1}}^2.
\end{equation}
As in \cref{rem:ID}, the $\dot{H}^{-1}$ norm is well-defined since the order-zero Fourier mode of its argument vanishes.

\medskip
Our second theorem (cf. \cite[Theorem 1.2, Proposition 3.1]{Puel2002}, \cite[Proposition 2.4]{GP2021}) shows that in the limit as $\hbar+\vep\rightarrow 0^+$ with independent rates, the density of the nonlinear Hartree solution converges to the uniform distribution while the current converges to a solution of the incompressible Euler equation. Again, the convergence is valid on an interval for which the incompressible Euler equation admits a classical solution. 

\begin{thm}\label{thm:mainH}
Let $d\in\{2,3\}$, $0<\al\leq 1$. Suppose that $T>0$ is such that $u\in L^\infty([0,T]; C^{1,\al}(\T^d))$ is a solution to the incompressible Euler equation \eqref{eq:Eul}. Given $\hbar,\vep>0$, let $R_{\hbar,\vep}^{0} \in \Dc(\H)$ satisfy the regularity condition \eqref{eq:iddmregH}. Let $R_{\hbar,\vep}$ denote the solution of the nonlinear Hartree equation \eqref{eq:Har} in $C(\R; \Dc(\H))$ with initial datum $R_{\hbar,\vep}^0$. Let $\rho_{\hbar,\vep}$ and $J_{\hbar,\vep}$ respectively denote the density and current of $R_{\hbar,\vep}$.

There exist constants $C_d, C_{d,\al}>0$ depending only on $d$ and $d,\al$, respectively, such that for every $0\leq t\leq T$,
\begin{equation}\label{eq:thmmainH}
\G_{\hbar,\vep}(R_{\hbar,\vep}^t,u^t) \leq \paren*{\G_{\hbar,\vep}(R_{\hbar,\vep}^0, u^0) + C_{d,\al}\vep^2\int_0^t \|u^\tau\|_{C^{1,\al}}^6d\tau} e^{C_d \int_0^t(1+ \|\nabla u^\tau\|_{L^\infty})d\tau}.
\end{equation}
Consequently, if $\{R_{\hbar,\vep}^0\}_{\hbar,\vep}$ is a sequence of initial data such that
\begin{equation}
\G_{\hbar,\vep}(R_{\hbar,\vep}^0, u^0) \xrightarrow[\hbar+\vep\rightarrow 0^+]{} 0,
\end{equation}
provided that $\vep=\vep(N)$ is chosen so that the rate \eqref{eq:eprate} holds, then for $1\leq \al\leq d$,
\begin{equation}\label{eq:thmmainHcnv}
\rho_{\hbar,\vep}^t \xrightharpoonup[\hbar +\vep\rightarrow 0^+]{} 1 \qquad \text{and} \qquad (J_{\hbar,\vep}^t)^{\al} \xrightharpoonup[\hbar+\vep\rightarrow 0^+]{} (u^t)^{\al}
\end{equation}
in the weak-* topology for the space $\M(\T^d)$ of signed Radon measures.
\end{thm}

\Cref{rem:quant,rem:reg,rem:Wig,rem:highd} all have their analogues for the 1-particle setting of \cref{thm:mainH}, which we will not state below. We briefly mention that the 1-particle quantum modulated energy estimate \eqref{eq:thmmainH} implies norm convergence of the density $\rho_{\hbar,\vep}^t$ to $1$ and the current $J_{\hbar,\vep}^t$ to $u^t$ in a negative-order fractional Sobolev space, which we show in \cref{ssec:MpfH}.

\section{Main proof}\label{sec:Mpf}
In this section, we prove \Cref{thm:main,thm:mainH}. The primary objective is to prove the Gronwall-type estimate \eqref{eq:thmmain} (cf. \cite[Proposition 3.7]{GP2021}, \cite[Equation (2.11)]{HkI2021}, \cite[Theorem 1.1]{Rosenzweig2021ne}) for the functional $\G_{\hbar,\vep,N}(R_{\hbar,\vep,N}^t,u^t)$ in the statement of \cref{thm:main}. The proof of the estimate \eqref{eq:thmmain} is spread out over \Cref{ssec:Mpftd,ssec:Mpffi,ssec:Mpfgron}. We then show in \cref{ssec:Mpfmain} that \eqref{eq:thmmain} implies norm convergence of the k-particle density and current in a certain norm on the k-fold tensor product of negative-order fractional Sobolev spaces, as claimed \cref{rem:quant}. From this Sobolev convergence, we deduce the weak-* convergence \eqref{eq:thmmaincnv}, completing the proof of \cref{thm:main}. Finally, we obtain in \cref{ssec:MpfH} the Gronwall-type estimate \eqref{eq:thmmainH} in the statement of \cref{thm:mainH} from the estimate \eqref{eq:thmmain}, and then use \eqref{eq:thmmainH} to show a similar Sobolev convergence assertion, along with the stated weak-* convergence \eqref{eq:thmmainHcnv}.

\subsection{Time derivative}\label{ssec:Mpftd}
We first use the incompressible Euler and $N$-body von Neumann equations \eqref{eq:Eul} and \eqref{eq:VN}, respectively, to compute the time derivative of the functional $\G_{\hbar,\vep,N}(R_{\hbar,\vep,N}^t,u^t)$. The main result is \cref{prop:MEtd} below (cf. \cite[Proposition 3.4]{GP2021}). We shall proceed formally, as this step is a variation on the analysis of Golse and Paul \cite[Section 3]{GP2021}, who carefully included all the necessary approximation steps so that the computation is justified. Although they only considered the case $\R^3$, their arguments may be adapted to treat $\T^d$, for $d\in \{2,3\}$. We leave the details as an exercise for the reader.

\begin{prop}\label{prop:MEtd}
The function $t\mapsto \G_{\hbar,\vep,N}(R_{\hbar,\vep,N}^t,u^t)$ is $C^1([0,T])$ and for $t\in [0,T]$, it holds that
\begin{multline}\label{eq:MEtd}
\frac{d}{dt}\G_{\hbar,\vep,N}(R_{\hbar,\vep,N,}^t,u^t) =-2\Tr_{\H}\paren*{\sqrt{R_{\hbar,\vep, N:1}^t}(\hbar D-u^t(X))\Sigma^t(X) (\hbar D-u^t(X)) \sqrt{R_{\hbar,\vep, N:1}^t}} \\
+\frac{1}{\vep^2}\int_{(\T^d)^N}\int_{(\T^d)^2\setminus\triangle} \paren*{u^t(x)-u^t(y)}\cdot\nabla V(x-y)d\paren*{\frac{1}{N}\sum_{j=1}^N\d_{x_j}-1-\vep^2\Uu^t}^{\otimes 2}(x,y)d\rho_{\hbar,\vep,N}^t(\ux_N)\\
- 2\int_{(\T^d)^N}\int_{\T^d}\div(-\D)^{-1}(u^t\Uu^t)x)d\paren*{\frac{1}{N}\sum_{j=1}^N\d_{x_j}-1-\vep^2\Uu^t}(x)d\rho_{\hbar,\vep,N}^t(\ux_N)\\
-2\int_{(\T^d)^N}\int_{\T^d}\p_t p^t(x)d\paren*{\frac{1}{N}\sum_{j=1}^N \d_{x_j}-1-\vep^2\Uu^t}(x)d\rho_{\hbar,\vep,N}^t(\ux_N),
\end{multline}
where $\Sigma^t(X)$ is the $d\times d$ matrix of operators whose entries are the multiplication operators $\Sigma_{\al\be}^t(X) \coloneqq \frac{1}{2}(\p_{x_\al} (u^t)^\be + \p_{x_\be} (u^t)^\al)(X)$.
\end{prop}
\begin{proof}[Proof of \cref{prop:MEtd}]
For the potential portion of the modulated energy, we split
\begin{multline}
\int_{(\T^d)^N}\Fr_N(\ux_N, 1+\vep^2\Uu^t)d\rho_{\hbar,\vep,N}^t(\ux_N) = \frac{(N-1)}{N}\int_{(\T^d)^2}V(x_1-x_2)d\rho_{\hbar,\vep,N:2}^t(x_1,x_2) \\
-2\int_{\T^d}(-\D)^{-1}(1+\vep^2\Uu^t)(x_1)d\rho_{\hbar,\vep,N:1}^t(x_1) +\int_{\T^d}(-\D)^{-1}(1+\vep^2\Uu^t)(x)d(1+\vep^2\Uu^t)(x),
\end{multline}
where above we have used \eqref{rem:pMEsym}. By the product rule, it follows that
\begin{multline}
\frac{d}{dt}\int_{(\T^d)^N}\Fr_N(\ux_N,1+\vep^2\Uu^t)d\rho_{\hbar,\vep,N}^t(\ux_N)= \frac{(N-1)}{N}\int_{(\T^d)^2}V(x_1-x_2)d\p_t \rho_{\hbar,\vep,N:2}^t(x_1,x_2) \\
-2\vep^2\int_{\T^d}(-\D)^{-1}(\p_t \Uu^t)(x_1)d\rho_{\hbar,\vep,N:1}^t(x_1) -2\int_{\T^d} \nabla(-\D)^{-1}(1+\vep^2\Uu^t)(x_1)\cdot d J_{\hbar,\vep,N}^t(x).
 \\
+2\vep^2\int_{(\T^d)^2}(-\D)^{-1}(\p_t\Uu^t)(x)d\paren*{1+\vep^2\Uu^t}(x).
\end{multline}
Here, we have used the fact that the density $\rho_{\hbar,\vep,N:1}^t$ of the 1-particle marginal $R_{\hbar,\vep,N:1}^t$ satisfies the equation \eqref{eq:rho}.

For the kinetic portion of the modulated energy, conservation of energy implies
\begin{equation}
\frac{d}{dt}\Tr_{\H}\paren*{\sqrt{R_{\hbar,\vep,N:1}^t}|\hbar D|^2 \sqrt{R_{\hbar,\vep,N:1}^t}}  =-\frac{(N-1)}{N\vep^2}\int_{(\T^d)^2}V(x_1-x_2)d\p_t\rho_{\hbar,\vep,N:2}^t(x_1,x_2).
\end{equation}
Using equation \eqref{eq:Eul} and \cite[Lemma 3.3]{GP2021}, we see that
\begin{multline}
\frac{d}{dt} \Tr_{\H}\paren*{\sqrt{R_{\hbar,\vep,N:1}^t}|u^t(X)|^2\sqrt{R_{\hbar,\vep,N:1}^t}} = -2\int_{\T^d} u^t(x_1)\cdot(u^t\cdot\nabla u^t)(x_1)d\rho_{\hbar,\vep,N:1}^t(x_1) \\
-2\int_{\T^d}u^t(x_1)\cdot\nabla p^t(x_1)d\rho_{\hbar,\vep,N:1}^t(x_1)  + 2\int_{\T^d}\nabla|u^t(x_1)|^2\cdot d J_{\hbar,\vep,N:1}^t(x_1)(x_1).
\end{multline}
For the cross-term computation, we extend the anticommutator notation $A\vee B \coloneqq AB+BA$ to vectors of operators $A=(A_j), B=(B_j)$, by $A\vee B\coloneqq A\cdot B - B\cdot A$. Using equation \eqref{eq:Eul} and \cite[Lemma 3.2]{GP2021},
\begin{multline}\label{eq:Hrep}
\frac{d}{dt}\Tr_{\H}\paren*{\sqrt{R_{\hbar,\vep,N:1}^t}(u^t(X)\vee \hbar D)\sqrt{R_{\hbar,\vep,N:1}^t}}\\
=-\frac{(N-1)}{\vep^2 N}\int_{(\T^d)^2} \paren*{u^t(x_1)-u^t(x_2)}\cdot \nabla V(x_1-x_2)d\rho_{\hbar,\vep,N:2}^t(x_1,x_2) \\
+ \Tr_{\H}\paren*{ \paren*{\paren*{\frac{\hbar}{2} D\vee \nabla u^t(X) - (u^t\cdot\nabla u^t)(X) - \nabla p^t(X)}\vee\hbar D}R_{\hbar,\vep,N:1}^t}.
\end{multline}
Lest there be any confusion, $(D\vee \nabla u^t(X))^{\al} = (D)^{\be}\vee\p_{x_{\be}} (u^t)^\al(X)$, where the summation over $\be$ is implicit.

Putting together all of the above computations, we arrive at
\begin{multline}\label{eq:dtGpre}
\frac{d}{dt}\G_{\hbar,\vep,N}(R_{\hbar,\vep,N}^t,u^t)= \frac{(N-1)}{\vep^2 N}\int_{(\T^d)^2} \paren*{u^t(x_1)-u^t(x_2)}\cdot \nabla V(x_1-x_2)d\rho_{\hbar,\vep,N:2}^t(x_1,x_2)\\
- \Tr_{\H}\paren*{\paren*{\paren*{\paren*{\frac{\hbar}{2} D\vee \nabla u^t(X)} - (u^t\cdot\nabla u^t)(X) - \nabla p^t(X)}\vee\hbar D}R_{\hbar,\vep,N:1}^t} \\
-2\int_{\T^d} u^t(x_1)\cdot(u^t\cdot\nabla u)^t(x_1)d\rho_{\hbar,\vep,N:1}^t(x_1) -2\int_{\T^d}u^t(x_1)\cdot\nabla p^t(x_1)d\rho_{\hbar,\vep,N:1}^t(x_1) \\
+ 2\int_{\T^d}\nabla|u^t(x_1)|^2\cdot dJ_{\hbar,\vep,N:1}^t(x_1)(x_1) -2\int_{\T^d}(-\D)^{-1}(\p_t \Uu^t)(x_1)d\rho_{\hbar,\vep,N:1}^t(x_1)\\
-\frac{2}{\vep^{2}}\int_{\T^d} \nabla(-\D)^{-1}(1+\vep^2\Uu^t)(x_1)\cdot dJ_{\hbar,\vep,N:1}^t(x_1) +2\int_{\T^d}(-\D)^{-1}(\p_t\Uu^t)(x)d(1+\vep^2\Uu^t)(x).
\end{multline}
Since $\rho_{\hbar,\vep,N}$ is a probability density and using its symmetry with respect to permutation of particle labels, we see that
\begin{multline}
-2\int_{\T^d} (-\D)^{-1}(\p_t\Uu^t)(x_1)d\rho_{\hbar,\vep,N:1}^t(x_1) + 2\int_{\T^d}(-\D)^{-1}(\p_t\Uu^t)(x)d\paren*{1+\vep^2\Uu^t}(x) \\
=-2\int_{(\T^d)^N}\int_{\T^d}\p_t\Uu^t(x)d\paren*{\frac{1}{N}\sum_{i=1}^N \d_{x_i}-1-\vep^2\Uu^t}(x)d\rho_{\hbar,\vep,N}^t(\ux_N).
\end{multline}
Note that since $-\D p = \Uu$, inverting the Laplacian, the preceding expression equals
\begin{equation}\label{eq:pressid}
-2\int_{(\T^d)^N}\int_{\T^d}\p_t p^t(x)d\paren*{\frac{1}{N}\sum_{i=1}^N \d_{x_i}-1-\vep^2\Uu^t}(x)d\rho_{\hbar,\vep,N}^t(\ux_N).
\end{equation}
By definition of $\rho_{\hbar,\vep,N:1}$ and the cyclicity of trace, we also have
\begin{multline}
-2\int_{\T^d}u^t(x_1)\cdot(u^t\cdot\nabla u^t)(x_1)d\rho_{\hbar,\vep,N:1}^t(x_1)\\
= -2\Tr_{\H}\paren*{\sqrt{R_{\hbar,\vep, N:1}^t} \paren*{u^t\cdot (u^t\cdot\nabla u^t)}(X)\sqrt{R_{\hbar,\vep, N:1}^t}}.
\end{multline}
By definition \eqref{eq:currdef} of the current, we also have the identities
\begin{align}
\Tr_{\H}\paren*{(\nabla p^t(X)\vee\hbar D)R_{\hbar,\vep,N:1}^t} = 2\int_{\T^d} \nabla p^t(x)\cdot dJ_{\hbar,\vep,N}^t(x), \label{eq:tdpid}\\
\Tr_{\H}\paren*{(\nabla|u^t|^2(X)\vee \hbar D)R_{\hbar,\vep, N:1}^t} = 2\int_{\T^d}\nabla|u(t,x)|^2\cdot dJ_{\hbar,\vep,N}^t(x).
\end{align}
Since trivially $\nabla(-\D)^{-1}(1) = 0$ and $- \D p = \Uu$, it follows from \eqref{eq:tdpid} that
\begin{equation}\label{eq:presscan}
\Tr_{\H}\paren*{(\nabla p^t(X)\vee\hbar D)R_{\hbar,\vep,N:1}^t} -\frac{2}{\vep^2}\int_{\T^d} \nabla(-\D)^{-1}(1+\vep^2\Uu^t)(x_1)\cdot dJ_{\hbar,\vep,N:1}^t(x_1)=0,
\end{equation}
which is the key cancellation provided by the corrector $\Uu$. Next, we again use the definition of $\rho_{\hbar,\vep,N:1}$ and cyclicity of trace to obtain
\begin{multline}
\Tr_{\H}\paren*{ \paren*{\paren*{\frac{\hbar}{2} D\vee \nabla u^t(X) - (u^t\cdot\nabla u^t)(X)}\vee\hbar D}R_{\hbar,\vep,N:1}^t} \\
=\Tr_{\H}\paren*{\sqrt{R_{\hbar,\vep,N:1}^t}\paren*{\paren*{\frac{\hbar}{2} D\vee \nabla u^t(X) - (u^t\cdot\nabla u^t)(X)}\vee\hbar D}\sqrt{R_{\hbar,\vep,N:1}^t}}
\end{multline}
and
\begin{equation}
\Tr_{\H}\paren*{(\nabla|u^t|^2(X)\vee \hbar D)R_{\hbar,\vep, N:1}^t} = \Tr_{\H}\paren*{\sqrt{R_{\hbar,\vep,N:1}^t}(\nabla|u^t|^2(X)\vee \hbar D)\sqrt{R_{\hbar,\vep,N:1}^t}}.
\end{equation}
As shown in the proof of \cite[Proposition 3.4, pg. 19]{GP2021} (note that $u$ is divergence-free in our setting), we also have the identity
\begin{multline}\label{eq:GPidke}
-\Tr_{\H}\Bigg(\sqrt{R_{\hbar,\vep, N:1}^t}\Bigg(\paren*{\frac{\hbar}{2} D\vee \nabla u^t(X) - (u^t\cdot\nabla)u^t(X)-\nabla|u^t|^2(X)}\vee\hbar D \\
+2\paren*{u^t\cdot (u^t\cdot\nabla u^t)}(X)\Bigg)\sqrt{R_{\hbar,\vep, N:1}^t}\Bigg) \\
=-2\Tr_{\H}\paren*{\sqrt{R_{\hbar,\vep, N:1}^t}\paren*{\hbar D-u^t(X)}\Sigma^t(X)\paren*{\hbar D-u^t(X)}\sqrt{R_{\hbar,\vep, N:1}^t}},
\end{multline}
where we have defined
\begin{equation}
\begin{split}
(\hbar D-u^t(X))\Sigma^t(X) (\hbar D-u^t(X)) \coloneqq (\hbar D-u^t(X))^\al\Sigma_{\al\be}^t(X) (\hbar D-u^t(X))^\be,\\
\Sigma_{\al\be}^t(X) \coloneqq \frac{1}{2}(\p_{x_\al} (u^t)^{\be} + \p_{x_\be} (u^t)^{\al})(X).
\end{split}
\end{equation}

As originally shown in \cite[Section 2]{HkI2021} (see the equation after (2.4)), we have the identity
\begin{multline}
\frac{1}{2\vep^2}\int_{(\T^d)^2\setminus\triangle} \paren*{u^t(x)-u^t(y)}\cdot\nabla V(x-y)d\paren*{\frac{1}{N}\sum_{i=1}^N\d_{x_i}-1-\vep^2\Uu^t}^{\otimes 2}(x,y) \\
= \frac{1}{2\vep^2N^2}\sum_{1\leq i\neq j\leq N} \paren*{u^t(x_i)-u^t(x_j)}\cdot\nabla V(x_i-x_j) - \frac{1}{N}\sum_{i=1}^N u^t(x_i)\cdot\nabla p^t(x_i) \\
+\int_{\T^d}\div(-\D)^{-1}(u^t\Uu^t)(x)d\paren*{\frac{1}{N}\sum_{i=1}^N\d_{x_i}-1-\vep^2\Uu^t}(x).
\end{multline}
Regarding $\ux_N$ as a random variable on $(\T^d)^N$ with law $\rho_{\hbar,\vep,N}^t$, we see from taking expectations of both sides of the preceding equality and using symmetry of $\rho_{\hbar,\vep,N}^t$ with respect to permutation of particle labels that
\begin{multline}\label{eq:HkIidrho}
\frac{1}{\vep^2}\int_{(\T^d)^N}\int_{(\T^d)^2\setminus\triangle} \paren*{u^t(x)-u^t(y)}\cdot\nabla V(x-y)d\paren*{\frac{1}{N}\sum_{i=1}^N\d_{x_i}-1-\vep^2\Uu^t}^{\otimes 2}(x,y)d\rho_{\hbar,\vep,N}^t(\ux_N)\\
=\frac{(N-1)}{\vep^2 N}\int_{(\T^d)^2}\paren*{u^t(x_1)-u^t(x_2)}\cdot\nabla V(x_1-x_2)d\rho_{\hbar,\vep,N:2}^t(x_1,x_2) \\
+ 2\int_{(\T^d)^N}\int_{\T^d}\div(-\D)^{-1}(u^t\Uu^t)(x)d\paren*{\frac{1}{N}\sum_{i=1}^N\d_{x_i}-1-\vep^2\Uu^t}(x)d\rho_{\hbar,\vep,N}^t(\ux_N) \\
-2\int_{\T^d} u^t(x_1)\cdot\nabla p^t(x_1)d\rho_{\hbar,\vep,N:1}^t(x_1).
\end{multline}

Applying the identities \eqref{eq:pressid},\eqref{eq:GPidke}, \eqref{eq:HkIidrho} together with the cancellation \eqref{eq:presscan} to the right-hand side of equality \eqref{eq:dtGpre}, we arrive at
\begin{multline}
\frac{d}{dt}\G_{\hbar,\vep,N}(R_{\hbar,\vep, N}^t,u^t) =-2\Tr_{\H}\paren*{\sqrt{R_{\hbar, \vep,N:1}^t}\paren*{\hbar D-u^t(X)}\Sigma^t(X)\paren*{\hbar D-u^t(X)} \sqrt{R_{\hbar, \vep,N:1}^t}} \\
+\frac{1}{\vep^2}\int_{(\T^d)^N}\int_{(\T^d)^2\setminus\triangle} \paren*{u^t(x)-u^t(y)}\cdot\nabla V(x-y)d\paren*{\frac{1}{N}\sum_{i=1}^N\d_{x_i}-1-\vep^2\Uu^t}^{\otimes 2}(x,y)d\rho_{\hbar,\vep,N}^t(\ux_N)\\
- 2\int_{(\T^d)^N}\int_{\T^d}\div(-\D)^{-1}(u^t\Uu^t)x)d\paren*{\frac{1}{N}\sum_{i=1}^N\d_{x_i}-1-\vep^2\Uu^t}(x)d\rho_{\hbar,\vep,N}^t(\ux_N) \\
-2\int_{(\T^d)^N}\int_{\T^d}\p_t p^t(x)d\paren*{\frac{1}{N}\sum_{i=1}^N \d_{x_i}-1-\vep^2\Uu^t}(x)d\rho_{\hbar,\vep,N}^t(\ux_N),
\end{multline}
which is exactly what we needed to show.
\end{proof}

\subsection{Functional inequalities}\label{ssec:Mpffi}
To bound the right-hand side of identity \eqref{eq:MEtd}, we need to recall some functional inequalities for the classical modulated energy of Serfaty. The first inequality, which we shall use to bound the last two terms in the right-hand side of \eqref{eq:MEtd}, essentially shows that the modulated potential energy $\Fr_N$ is coercive. Such estimates are by now well-known. The version we state below is from \cite[Proposition 3.8]{Rosenzweig2021ne}, which is a variation on an earlier estimate \cite[Proposition 3.6]{Serfaty2020}.

\begin{prop} \label{prop:coer}
There exists a constant $C_d>0$ such that for any functions $\phi\in W^{1,\infty}(\T^d)$ and $\mu\in L^\infty(\T^d)$ and pairwise distinct configuration $\ux_N \in (\T^d)^N$, it holds that
\begin{multline}
\left|\int_{\T^d}\phi(x)d\paren*{\frac{1}{N}\sum_{j=1}^N\d_{x_j} - \mu}(x)\right| \leq C_d\|\nabla\phi\|_{L^\infty}N^{-1/d} \\
+\|\nabla\phi\|_{L^2}\paren*{\Fr_N(\ux_N,\mu) + C_d(1+\|\mu\|_{L^\infty})\paren*{\frac{1+(\log N)\indic_{d=2}}{N^{2/d}}} }^{1/2}.
\end{multline}
\end{prop}

The second functional inequality we need is to control the second term in the right-hand side of \eqref{eq:MEtd} in terms of the expected modulated potential energy plus some terms which are $o(1)$ as $N\rightarrow\infty$. A nonsharp estimate, where the sharpness is measured in terms of the size of the additive $N$-dependent error terms, for such expressions was originally proven by Serfaty in \cite[Proposition 1.1]{Serfaty2020} using a renormalization procedure to re-express this expression in terms of a stress-energy tensor. Serfaty later proved the (believed) sharp estimate, which we state below, in \cite[Proposition 4.1]{Serfaty2021} (see also \cite[Proposition 3.9]{Rosenzweig2021ne} for a more direct proof by the author).

\begin{prop}\label{prop:comm}
There exists a constant $C_d>0$ such that for any Lipschitz vector field $v:\T^d\rightarrow\R^d$, function $\mu \in L^\infty(\T^d)$, and pairwise distinct configuration $\ux_N \in (\T^d)^N$, it holds that
\begin{multline}
\left|\int_{(\T^d)^2\setminus\triangle} \paren*{v(x)-v(y)}\cdot \nabla V(x-y)d\paren*{\frac{1}{N}\sum_{j=1}^N\d_{x_j} - \mu}^{\otimes 2}(x,y)\right| \\
\leq C_d \|\nabla v\|_{L^\infty}\paren*{\Fr_N(\ux_N,\mu) + C_d(1+\|\mu\|_{L^\infty})\paren*{\frac{1+(\log N)\indic_{d=2}}{N^{2/d}}} }.
\end{multline}
\end{prop}

We remark that \cref{prop:comm} has an improved dependence on $N$ for the additive error terms compared to the earlier estimate \cite[Proposition 1.1]{Serfaty2020} invoked by Golse and Paul \cite{GP2021}. As commented several times before, the sharper estimate of \cref{prop:comm} is necessary to obtain the scaling relation for $\vep,N$ presented in \cref{thm:main}. Such a relation, of course, is not present in the work of Golse-Paul, since they consider a different asymptotic regime.

\subsection{Gronwall argument}\label{ssec:Mpfgron}
With the identity \eqref{eq:MEtd} from the \cref{ssec:Mpftd} and the functional inequalities reviewed in \cref{ssec:Mpffi}, let us now close our Gronwall argument for the modulated energy $\G_{\hbar,\vep,N}(R_{\hbar,\vep,N}^t,u^t)$, thereby proving the inequality \eqref{eq:thmmain}.

\begin{proof}[Proof of estimate \eqref{eq:thmmain}]
We start by estimating each of the terms in the right-hand side of \eqref{eq:MEtd}. All of the estimates below are static (i.e. pointwise in time), therefore we suppress the time dependence to simplify the notation.

\medskip
\textbullet By Cauchy-Schwarz,
\begin{multline}\label{eq:MEft1}
\left|\Tr_{\H}\paren*{\sqrt{R_{\hbar,\vep, N:1}}(\hbar D-u(X))\Sigma (\hbar D-u(X)) \sqrt{R_{\hbar,\vep, N:1}}}\right| \\
\lesssim_d \|\nabla u\|_{L^\infty} \Tr_{\H}\paren*{\sqrt{R_{\hbar,\vep, N:1}}|\hbar D-u(X)|^2\sqrt{R_{\hbar,\vep, N:1}}}.
\end{multline}

\textbullet Applying \cref{prop:comm} with $v=u$, $\mu = 1+\vep^2\Uu$, and fixed $\ux_N\in (\T^d)^N$ with distinct components, we find
\begin{multline}
\left|\int_{(\T^d)^2\setminus\triangle} \paren*{u(x)-u(y)}\cdot\nabla V(x-y)d\paren*{\frac{1}{N}\sum_{j=1}^N\d_{x_j}-1-\vep^2\Uu}^{\otimes 2}(x,y)\right| \\
\leq  C_d\|\nabla u\|_{L^\infty}\paren*{\Fr_N(\ux_N,1+\vep^2\Uu) +C_d (1+\vep^2\|\nabla u\|_{L^\infty}^2)\frac{(1+(\log N)\indic_{d=2})}{N^{2/d}}}.
\end{multline}
Since $\rho_{\hbar,\vep,N}$ is a probability density, it follows that
\begin{multline}\label{eq:MEft2}
\frac{1}{\vep^{2}}\left|\int_{(\T^d)^N}\int_{(\T^d)^2\setminus\triangle} \paren*{u(x)-u(y)}\cdot\nabla V(x-y)d\paren*{\frac{1}{N}\sum_{j=1}^N\d_{x_j}-1-\vep^2\Uu}^{\otimes 2}(x,y)d\rho_{\hbar,\vep,N}(\ux_N)\right| \\
\leq \frac{C_d\|\nabla u\|_{L^\infty}}{\vep^2}\paren*{C_d(1+\vep^2\|\nabla u\|_{L^\infty}^2)\frac{(1+(\log N)\indic_{d=2})}{N^{2/d}} + \int_{(\T^d)^N} \Fr_N(\ux_N,1+\vep^2\Uu)d\rho_{\hbar,\vep,N}(\ux_N)}.
\end{multline}

\textbullet Lastly, applying \cref{prop:coer} with $\mu=1+\vep^2\Uu$ and $\phi = \nabla(-\D)^{-1}(u\Uu)$ and $\phi = \p_t p$, successively, it follows that
\begin{multline}\label{eq:t3exp}
\left|\int_{\T^d} \div(-\D)^{-1}(u\Uu)(x)d\paren*{\frac{1}{N}\sum_{j=1}^N\d_{x_j} - 1-\vep^2\Uu}(x)\right| \leq C_d\|\nabla\div(-\D)^{-1}(u\Uu)\|_{L^\infty} N^{-1/d} \\
+ C_d\|\nabla\div(-\D)^{-1}(u\Uu)\|_{L^2}\paren*{\Fr_N(\ux_N,1+\vep^2\Uu) + C_d(1+\vep^2\|\nabla u\|_{L^\infty}^2)\frac{(1+(\log N)\indic_{d=2})}{N^{2/d}}}^{1/2}
\end{multline}
and
\begin{multline}\label{eq:t4exp}
\left|\int_{\T^d} \p_t p(x)d\paren*{\frac{1}{N}\sum_{j=1}^N\d_{x_j} - 1-\vep^2\Uu}(x)\right| \leq C_d\|\nabla\p_t p\|_{L^\infty} N^{-1/d} \\
+ C_d\|\nabla\p_t p\|_{L^2}\paren*{\Fr_N(\ux_N,1+\vep^2\Uu) + C_d(1+\vep^2\|\nabla u\|_{L^\infty}^2)\frac{(1+(\log N)\indic_{d=2})}{N^{2/d}}}^{1/2}.
\end{multline}
Writing $1=\vep^{-1}\vep$ and using the elementary inequality $ab\leq (a^2+b^2)/2$, it follows that the preceding right-hand sides are respectively bounded by
\begin{multline}
C_d\vep^2\|\nabla\div(-\D)^{-1}(u\Uu)\|_{L^2}^2 + C_d\|\nabla\div(-\D)^{-1}(u\Uu)\|_{L^\infty} N^{-1/d}\\
+ \frac{\Fr_N(\ux_N,1+\vep^2\Uu)}{\vep^2} + \frac{C_d(1+\vep^2\|\nabla u\|_{L^\infty}^2)}{\vep^2}\paren*{\frac{1+(\log N)\indic_{d=2}}{N^{2/d}}}
\end{multline}
and
\begin{multline}
C_d\vep^2\|\nabla\p_t p\|_{L^2}^2 + C_d\|\nabla\p_t p\|_{L^\infty} N^{-1/d}+ \frac{\Fr_N(\ux_N,1+\vep^2\Uu)}{\vep^2} \\
 + \frac{C_d(1+\vep^2\|\nabla u\|_{L^\infty}^2)}{\vep^2}\paren*{\frac{1+(\log N)\indic_{d=2}}{N^{2/d}}}.
\end{multline}
By elliptic regularity and the algebra property of H\"older spaces (for example, see the proof of \cite[Lemma 4.3]{Rosenzweig2021ne} for details),
\begin{equation}
\|\nabla\div(-\D)^{-1}(u\Uu)\|_{L^\infty} \lesssim_{d,\al} \|u\|_{C^{1,\al}}^3
\end{equation}
for any $0<\al\leq 1$. By direct computation of the equation satisfied by $\p_t p$ using \eqref{eq:Eul} (see \cite[Equation 4.10]{Rosenzweig2021ne}), then again using elliptic regularity and the algebra property of H\"older spaces, we also have
\begin{equation}
\|\nabla\p_t p\|_{L^\infty} \lesssim_{d,\al}\|u\|_{C^{1,\al}}^3.
\end{equation}
Using that $\|\cdot\|_{L^2}\leq \|\cdot\|_{L^\infty}$, taking expectations of both sides of inequalities \eqref{eq:t3exp} and \eqref{eq:t4exp} with respect to the law $\rho_{\hbar,\vep,N}$ and using that, modulo a set of Lebesgue measure zero, the configurations $\ux_N$ are pairwise distinct, we then obtain
\begin{multline}\label{eq:MEft3}
\left|\int_{(\T^d)^N}\int_{\T^d}\paren*{\div(-\D)^{-1}(u\Uu)(x) + \p_t p (x)}d\paren*{\frac{1}{N}\sum_{j=1}^N\d_{x_j} - 1-\vep^2\Uu}(x)d\rho_{\hbar,\vep,N}(\ux_N)\right| \\
\leq C_{d,\al}\paren*{\vep^2\|u\|_{C^{1,\al}}^3+N^{-1/d}}\|u\|_{C^{1,\al}}^3 + \frac{C_d(1+\vep^2\|\nabla u\|_{L^\infty}^2)(1+(\log N)\indic_{d=2})}{\vep^2 N^{2/d}} \\
+ \frac{1}{\vep^{2}}\int_{(\T^d)^N}\Fr_N(\ux_N,1+\vep^2\Uu)d\rho_{\hbar,\vep,N}(\ux_N).
\end{multline}

\medskip

Putting together the estimates \eqref{eq:MEft1}, \eqref{eq:MEft2}, and \eqref{eq:MEft3}, then integrating with respect to time over the interval $[0,t]$, we arrive at
\begin{multline}
\left|\G_{\hbar,\vep,N}(R_{\hbar,\vep,N}^t,u^t)\right| \leq \left|\G_{\hbar,\vep,N}(R_{\hbar,\vep,N}^0, u^0)\right|  + C_{d,\al}\int_0^t \paren*{\vep^2\|u^\tau\|_{C^{1,\al}}^3+N^{-1/d}}\|u^\tau\|_{C^{1,\al}}^3 d\tau \\
+ C_d\int_0^t\|\nabla u^\tau\|_{L^\infty} \Tr_{\H}\paren*{\sqrt{R_{\hbar,\vep,N:1}^\tau}|\hbar D-u^\tau(X)|^2\sqrt{R_{\hbar,\vep,N:1}^\tau}} d\tau \\
+C_d\int_0^t \frac{(1+\|\nabla u^\tau\|_{L^\infty})}{\vep^2}\Bigg(\frac{C_d(1+(\log N)\indic_{d=2})(1+\vep^2\|\nabla u^\tau\|_{L^\infty}^2)}{ N^{2/d}}  \\
 + \int_{(\T^d)^N}\Fr_N(\ux_N,1+\vep^2\Uu^\tau)d\rho_{\hbar,\vep,N}^\tau(\ux_N)\Bigg)d\tau.
\end{multline}
Recalling the definition \eqref{eq:qme} of $\G_{\hbar,\vep,N}$ and \cref{rem:MElb}, the right-hand side is controlled by
\begin{multline}
\left|\G_{\hbar,\vep,N}(R_{\hbar,\vep,N}^0, u^0)\right| + C_{d,\al}\int_0^t \paren*{\vep^2\|u^\tau\|_{C^{1,\al}}^3+N^{-1/d}}\|u^\tau\|_{C^{1,\al}}^3 d\tau\\
+C_d\int_0^t (1+\|\nabla u^\tau\|_{L^\infty})\paren*{\G_{\hbar,\vep,N}(R_{\hbar,\vep,N}^\tau,u^\tau) + \frac{C_d(1+(\log N)\indic_{d=2})(1+\vep^2\|\nabla u^\tau\|_{L^\infty}^2)}{\vep^2 N^{2/d}}}d\tau.
\end{multline}
Applying the Gronwall-Bellman lemma, we arrive at the desired estimate \eqref{eq:thmmain}.
\end{proof}

\subsection{Proof of \cref{thm:main}}\label{ssec:Mpfmain}
Let us now complete the proof of \cref{thm:main} by showing how the quantum modulated energy estimate of \eqref{eq:thmmain} implies the remaining weak-* convergence assertion \eqref{eq:thmmaincnv}. To do this, one could use the analysis of \cite[Section 4]{GP2021}. Instead, we present an argument which we feel more closely parallels the analysis used to show the classical modulated energy controls Sobolev convergence. As a consequence, we obtain strong, quantitative convergence in negative-order Sobolev spaces. 

We first deal with the convergence of the k-particle densities $\rho_{\hbar,\vep,N:k}^t$. Let $\{\ux_N^t\}_{t\geq 0}$ be a process with values in $(\T^d)^N$ whose time marginals are given by $\rho_{\hbar,\vep,N}^t$. Note that $\rho_{\hbar,\vep,N}^t$-a.s., it holds that $x_i^t\neq x_j^t$ for every $1\leq i\neq j\leq N$. Let $\mu_N^t \coloneqq \frac{1}{N}\sum_{j=1}^N \d_{x_j^t}$ denote the empirical measure associated to $\ux_N^t$. Using the coercivity estimate of \cref{prop:coer}, there is a constant $C_d>0$ such that for any $\phi\in W^{1,\infty}(\T^d)$,
\begin{multline}\label{eq:coer1}
\left|\int_{\T^d}\phi(x)d\paren*{\mu_N^t - 1-\vep^2\Uu^t}(x)\right| \leq   C_d\|\nabla\phi\|_{L^\infty} N^{-1/d} \\
+\|\nabla\phi\|_{L^2}\paren*{\Fr_N(\ux_N^t,1+\vep^2 \Uu^t) + C_d(1+\vep^2\|\nabla u^t\|_{L^\infty}^2)\paren*{\frac{1+(\log N)\indic_{d=2}}{N^{2/d}}}}^{1/2}.
\end{multline}
By Sobolev embedding, it follows from the estimate \eqref{eq:coer1} and duality that for any $1<p<\infty$ and $s<-(d+1)+\frac{d}{p}$,
\begin{equation}
\left\|\mu_N^t - 1-\vep^2\Uu^t\right\|_{W^{s,p}}^2 \lesssim_{d,s,p} \Fr_N(\ux_N^t,1+\vep^2 \Uu^t) + \frac{C_d(1+\vep^2\|\nabla u^t\|_{L^\infty}^2)(1+(\log N)\indic_{d=2})}{N^{2/d}},
\end{equation}
where $W^{s,p}$ denotes the usual Bessel potential space (e.g., see \cite[Section 1.3.1]{Grafakos2014m}). Taking the expectation of both sides of the preceding inequality, we then obtain
\begin{multline}\label{eq:Hsexp}
\int_{(\T^d)^N}\left\|\mu_N - 1-\vep^2\Uu^t\right\|_{W^{s,p}}^2 d\rho_{\hbar,\vep,N}^t(\ux_N) \lesssim_{d,s,p} \int_{(\T^d)^N} \Fr_N(\ux_N,1+\vep^2 \Uu^t)d\rho_{\hbar,\vep,N}^t(\ux_N) \\
+  \frac{C_d(1+\vep^2\|\nabla u^t\|_{L^\infty}^2)(1+(\log N)\indic_{d=2})}{N^{2/d}},
\end{multline}
where $\mu_N\coloneqq \frac{1}{N}\sum_{j=1}^N \d_{x_j}$. Next, using \cite[(7.20), (7.21)]{RS2016}, we see that for any symmetric test function $\phi\in C^\infty((\T^d)^k)$,
\begin{multline}\label{eq:coerk}
\left|\int_{(\T^d)^k}\phi(\ux_k) d\paren*{\rho_{\hbar,\vep,N:k}^t-(1+\vep^2\Uu^t)^{\otimes k}}(\ux_k) \right| \leq Ck\paren*{1+\vep^2\|\nabla u\|_{L^\infty}^2}^{k-1}\\
\times\sup_{\ux_{k-1}\in (\T^d)^{k-1}} \|\phi(\ux_{k-1},\cdot)\|_{W^{-s,p'}}\int_{(\T^d)^N} \left\|\mu_N-1-\vep^2\Uu^t \right\|_{W^{s,p}} d\rho_{\hbar,\vep,N}^t(\ux_N)
\end{multline}
for any index $s<-(d+1)+\frac{d}{p}$, where $\frac{1}{p'}+\frac{1}{p}=1$. Note that by Cauchy-Schwarz and using that $\rho_{\hbar,\vep,N}^t$ is a probability density, the preceding right-hand side is $\leq$
\begin{multline}
Ck\paren*{1+\vep^2\|\nabla u\|_{L^\infty}^2}^{k-1}\sup_{\ux_{k-1}\in (\T^d)^{k-1}} \|\phi(\ux_{k-1},\cdot)\|_{W^{-s,p'}}  \\
\times\paren*{\int_{(\T^d)^N} \left\|\mu_N-1-\vep^2\Uu^t \right\|_{W^{s,p}}^2 d\rho_{\hbar,\vep,N}^t(\ux_N)}^{1/2}.
\end{multline}
Combining this bound with \eqref{eq:Hsexp}, it follows that
\begin{multline}\label{eq:tpdenssup}
\left|\int_{(\T^d)^k}\phi(\ux_k) d\paren*{\rho_{\hbar,\vep,N:k}^t-(1+\vep^2\Uu^t)^{\otimes k}}(\ux_k) \right| \lesssim_{d,s,p} k\paren*{1+\vep^2\|\nabla u\|_{L^\infty}^2}^{k-1}\\
\times\sup_{\ux_{k-1}\in (\T^d)^{k-1}} \|\phi(\ux_{k-1},\cdot)\|_{W^{-s,p'}}  \Bigg(\int_{(\T^d)^N} \Fr_N(\ux_N,1+\vep^2 \Uu^t)d\rho_{\hbar,\vep,N}^t(\ux_N) \\
+  \frac{C_{d}(1+\vep^2\|\nabla u^t\|_{L^\infty}^2)(1+(\log N)\indic_{d=2})}{N^{2/d}}\Bigg)^{1/2}.
\end{multline}
For $r\in\R$ and $1<q<\infty$, let $(W^{r,q})^{\otimes_{\al_q} k}$ denote the completion with respect to the $q$-nuclear tensor norm of the k-fold algebraic tensor product of the space $W^{r,q}$. This space may be identified with the space of $\phi\in\Dc'((\T^d)^k)$ such that
\begin{equation}
\|\jp{\nabla_{x_1}}^{r}\cdots\jp{\nabla_{x_k}}^{r}\phi\|_{L^{q}((\T^d)^k)} < \infty,
\end{equation}
where $\jp{\nabla_{x_j}}^{r}$ is the Fourier multiplier with Japanese bracket symbol $\jp{2\pi\xi_j}^{r}$. We refer to \cite{SU2009} for details on tensor products of Sobolev spaces. Taking the supremum over $\|\phi\|_{(W^{-s,p'})^{\otimes_{\al_{p'}} k}} \leq 2$ in both sides of \eqref{eq:tpdenssup} and using that $(W^{-s,p'})^{\otimes_{\al_{p'}} k}$ embeds in $C((\T^d)^{k-1}; W^{-s,p'}(\T^d))$ by choice of $s$,
\begin{multline}
\left\|\rho_{\hbar,\vep,N:k}^t-(1+\vep^2\Uu^t)^{\otimes k}\right\|_{(W^{s,p})^{\otimes_{\al_p} k}} \lesssim_{d,s,p,k} \paren*{1+\vep^2\|\nabla u\|_{L^\infty}^2}^{k-1}\\
\times\Bigg(\int_{(\T^d)^N} \Fr_N(\ux_N,1+\vep^2 \Uu^t)d\rho_{\hbar,\vep,N}^t(\ux_N) +  \frac{C_{d}(1+\vep^2\|\nabla u^t\|_{L^\infty}^2)(1+(\log N)\indic_{d=2}) }{N^{2/d}} \Bigg)^{1/2}.
\end{multline}
From the triangle inequality and the embedding $L^\infty \subset W^{s,p}$, it then follows from the preceding estimate that
\begin{multline}\label{eq:rho1diff}
\left\|\rho_{\hbar,\vep, N:k}^t - 1\right\|_{(W^{s,p})^{\otimes_{\al_p} k}} \lesssim_{d,s,p,k} \vep^2\|\nabla u^t\|_{L^\infty}^2\paren*{1+(\vep^2 \|\nabla u^t\|_{L^\infty})^{k-1} } \\
+\paren*{1+\vep^2\|\nabla u\|_{L^\infty}^2}^{k-1}\Bigg(\int_{(\T^d)^N} \Fr_N(\ux_N,1+\vep^2 \Uu^t)d\rho_{\hbar,\vep,N}^t(\ux_N) \\
 +  \frac{C_{d}(1+\vep^2\|\nabla u^t\|_{L^\infty}^2) (1+(\log N)\indic_{d=2})}{N^{2/d}} \Bigg)^{1/2}.
\end{multline}
Using the estimate \eqref{eq:thmmain} to bound $\Fr_N(\ux_N,1+\vep^2\Uu^t)$ in the right-hand side yields the desired convergence in $(W^{s,p})^{\otimes_{\al_p} k}$ norm, provided that $1<p<\infty$ and $s<\frac{d}{p}-d-1$. This convergence also implies convergence in the weak-* topology for $\M((\T^d)^k)$ by using the density of $W^{-s,p'}\subset C(\T^d)$.

\bigskip
Proving convergence of the current $J_{\hbar,\vep,N:k}^t$ is simpler, and we can follow---albeit with modification to obtain rates of convergence---the general argument of \cite[Subsection 4.3]{GP2021}, which treated the $k=1$ case. To this end, we recall that the current $J_{\hbar,\vep,N:k}$ of the k-particle marginal $R_{\hbar,\vep,N:k}$ is, by definition, a vector with components $(J_{\hbar,\vep,N:k})_{j}$ in $(\M((\T^d)^k))^{d}$, for $j=1,\ldots,k$. In other words, for fixed $j$, $(J_{\hbar,\vep,N:k})_j = \paren*{(J_{\hbar,\vep,N:k})_j^{\al}}_{\al=1}^d$, where each $(J_{\hbar,\vep,N:k})_j^\al\in \M((\T^d)^k)$. $J_{\hbar,\vep,N:k}$ is the continuous linear functional on the Banach space $C((\T^d)^k; (\R^d)^k)$ of continuous bounded vector fields $v=(v_j)_{j=1}^k$, with $v_j=(v_j^\al)_{\al=1}^d \in C((\T^d)^k; \R^d)$,  defined by
\begin{equation}
\int_{(\T^d)^k} v(\ux_k)\cdot dJ_{\hbar,\vep,N:k}(\ux_k)= \Tr_{\H_k}\paren*{v(\ul{X}_k) \cdot \paren*{R_{\hbar,\vep,N:k} \vee \frac{\hbar}{2} D_{\ux_k}}}.
\end{equation}
Above, $D_{\ux_k} \coloneqq -i\nabla_{\ux_k}$, with $\nabla_{\ux_k} \coloneqq (\nabla_{x_1},\ldots,\nabla_{x_k})$ with each $\nabla_{x_j}$ denoting the gradient with respect to the variable $x_j\in \R^d$. Let $(u^t)^{\otimes k}(\ux_k) \coloneqq (u^t(x_1),\ldots,u^t(x_k)) \in (\R^d)^k$. We let $(u^t)^{\otimes k}(\ul{X}_k)$ denote the induced vector of multiplication operators. By cyclicity of trace, the commutativity of multiplication operators, and the definition of the k-particle density $\rho_{\hbar,\vep,N:k}^t$, we have that
\begin{equation}
\frac{1}{2}\Tr_{\H_k}\paren*{v(\ul{X}_k)\cdot \paren*{R_{\hbar,\vep,N:k}^t \vee (u^t)^{\otimes k}(\ul{X}_k)}} = \int_{(\T^d)^k} v(\ux_k)\cdot (u^t)^{\otimes k}(\ux_k)d\rho_{\hbar,\vep,N:k}^t(\ux_k).
\end{equation}
Hence,
\begin{multline}
\int_{(\T^d)^k}v(\ux_k)\cdot d\paren*{J_{\hbar,\vep,N:k}^t-\rho_{\hbar,\vep,N:k}^t(u^t)^{\otimes k}}(\ux_k) \\
= \frac{1}{2}\Tr_{\H_k}\paren*{v(\ul{X}_k) \cdot \paren*{R_{\hbar,\vep,N:k}^t \vee \paren*{\hbar D_{\ux_k}-(u^t)^{\otimes k}(\ul{X}_k)}}}.
\end{multline}
By Cauchy-Schwarz, the magnitude of the right-hand side is $\lesssim_d$
\begin{multline}
\sum_{j=1}^k \left\|v_j(\ul{X}_k)\sqrt{R_{\hbar,\vep,N:k}^t}\right\|_2 \left\|\sqrt{R_{\hbar,\vep,N:k}^t}\paren*{\hbar D_{x_j}-u^{t}(X_j)} \right\|_2 \\
\lesssim_d \|v\|_{L^\infty}\sum_{j=1}^k  \left\|\sqrt{R_{\hbar,\vep,N:k}^t}\paren*{\hbar D_{x_j}-u^{t}(X_j)} \right\|_2.
\end{multline}
Above, we have used that $R_{\hbar,\vep,N:k}^t$ has unit trace. Writing the Hilbert-Schmidt norm as the square root of a trace, we see that
\begin{align}
\left\|\sqrt{R_{\hbar,\vep,N:k}^t}\paren*{\hbar D_{x_j}-u^{t}(X_j)} \right\|_2^2 \leq \Tr_{\H_k}\paren*{\sqrt{R_{\hbar,\vep,N:k}^t}\left|\hbar D_{x_j}-u^t(X_j)\right|^2\sqrt{R_{\hbar,\vep,N:k}^t}} \nn\\
= \Tr_{\H}\paren*{\sqrt{R_{\hbar,\vep,N:1}^t}\left|\hbar D-u^t(X)\right|^2\sqrt{R_{\hbar,\vep,N:1}^t}},
\end{align}
where the ultimate line follows from symmetry with respect to permutation of particle labels. After a little bookkeeping, we realize we have shown
\begin{multline}
\left|\int_{(\T^d)^k}v(\ux_k)\cdot d\paren*{J_{\hbar,\vep,N:k}^t-\rho_{\hbar,\vep,N:k}^t(u^t)^{\otimes k}}(\ux_k)  \right| \\
\lesssim_d k \|v\|_{L^\infty}\paren*{\Tr_{\H}\paren*{\sqrt{R_{\hbar,\vep,N:1}^t}\left|\hbar D-u^t(X)\right|^2\sqrt{R_{\hbar,\vep,N:1}^t}} }^{1/2}
\end{multline}
for any $v\in C((\T^d)^k; (\R^d)^k)$. Evidently, the right-hand side of the preceding inequality is controlled by the modulated energy $\G_{\hbar,\vep,N}(R_{\hbar,\vep,N}^t,u^t)$. Taking the supremum over all such vector fields with $\|v\|_{L^\infty} \leq 1$, implies that
\begin{equation}
\left\|J_{\hbar,\vep,N:k}^t-\rho_{\hbar,\vep,N:k}^t(u^t)^{\otimes k}\right\|_{\M} \lesssim_d k \paren*{\Tr_{\H}\paren*{\sqrt{R_{\hbar,\vep,N:1}^t}\left|\hbar D-u^t(X)\right|^2\sqrt{R_{\hbar,\vep,N:1}^t}} }^{1/2},
\end{equation}
where $\M = \M((\T^d)^k)$. By Sobolev embedding, we also have
\begin{equation}\label{eq:Jrhoudiff}
\left\|J_{\hbar,\vep,N:k}^t-\rho_{\hbar,\vep,N:k}^t(u^t)^{\otimes k}\right\|_{(W^{s,p})^{\otimes_{\al_p} k}} \lesssim_{d,s,p,k} \paren*{\Tr_{\H}\paren*{\sqrt{R_{\hbar,\vep,N:1}^t}\left|\hbar D-u^t(X)\right|^2\sqrt{R_{\hbar,\vep,N:1}^t}} }^{1/2}
\end{equation}
for any $1<p<\infty$ and $s<\frac{d}{p}-d$. Above, the norms are, with a slight abuse of notation, for vector-valued measures/distributions. By the triangle inequality,
\begin{multline}\label{eq:Judiff}
\left\|J_{\hbar,\vep,N:k}^t-(u^t)^{\otimes k}\right\|_{(W^{s,p})^{\otimes_{\al_p} k}} \leq \left\|J_{\hbar,\vep,N:k}^t-\rho_{\hbar,\vep,N:k}^t(u^t)^{\otimes k}\right\|_{(W^{s,p})^{\otimes_{\al_p} k}} \\
+ \left\|(\rho_{\hbar,\vep,N:k}^t-1)(u^t)^{\otimes k}\right\|_{(W^{s,p})^{\otimes_{\al_p} k}}.
\end{multline}
For any test function $\phi\in C^\infty((\T^d)^k)$, we observe from Sobolev embedding that
\begin{equation}
\|\phi (u^t)^{\otimes k}\|_{(W^{-s,p'})^{\otimes_{\al_{p'}} k}} \lesssim_{d,s,p,k}  \|\phi\|_{(W^{-s,p'})^{\otimes_{\al_{p'}} k}}\|u^t\|_{W^{-s,p'}}^k \lesssim_{d,s,p,\al,k} \|\phi\|_{(W^{-s,p'})^{\otimes_{\al_{p'}} k}}\|u^t\|_{C^{1,\al}}^k,
\end{equation}
provided that $1+\al>-s>d-\frac{d}{p}$. Such an $s$ exists provided that we choose $p<\frac{d}{(d-1-\al)_+}$, where $(\cdot)_+\coloneqq \max\{0,\cdot\}$. It follows now from duality that
\begin{equation}
\left\|(\rho_{\hbar,\vep,N:k}^t-1)(u^t)^{\otimes k}\right\|_{(W^{s,p})^{\otimes_{\al_p} k}} \lesssim_{d,s,p,\al,k} \|u^t\|_{C^{1,\al}}^k \left\|\rho_{\hbar,\vep,N:k}^t-1\right\|_{(W^{s,p})^{\otimes_{\al_p} k}}.
\end{equation}
Assuming that we also have $s<\frac{d}{p}-d-1$, which can be ensured by choosing $p<\frac{d}{(d-\al)_+}$, we may apply the estimate \eqref{eq:rho1diff} to find that the preceding right-hand side is $\lesssim_{d,s,p,k}$
\begin{multline}
\|u^t\|_{C^{1,\al}}^k\Bigg(\vep^2\|\nabla u^t\|_{L^\infty}^2\paren*{1+(\vep^2 \|\nabla u^t\|_{L^\infty})^{k-1} } +\paren*{1+\vep^2\|\nabla u\|_{L^\infty}^2}^{k-1}\\
\times\Bigg(\int_{(\T^d)^N} \Fr_N(\ux_N,1+\vep^2 \Uu^t)d\rho_{\hbar,\vep,N}^t(\ux_N)  +  \frac{C_{d}(1+\vep^2\|\nabla u^t\|_{L^\infty}^2) (1+(\log N)\indic_{d=2})}{N^{2/d}} \Bigg)^{1/2}\Bigg).
\end{multline}
Applying the preceding estimate and the estimate \eqref{eq:Jrhoudiff} to the right-hand side of \eqref{eq:Judiff}, we conclude that for any $1<p<\frac{d}{(d-\al)+}$ and $s<\frac{d}{p}-d-1$,
\begin{multline}
\left\|J_{\hbar,\vep,N:k}^t-(u^t)^{\otimes k}\right\|_{(W^{s,p})^{\otimes_{\al_p} k}} \lesssim_{d,s,p,\al,k} \paren*{\Tr_{\H}\paren*{\sqrt{R_{\hbar,\vep,N:1}^t}\left|\hbar D-u^t(X)\right|^2\sqrt{R_{\hbar,\vep,N:1}^t}} }^{1/2} \\
+ \|u^t\|_{C^{1,\al}}^k\Bigg(\vep^2\|\nabla u^t\|_{L^\infty}^2\paren*{1+(\vep^2 \|\nabla u^t\|_{L^\infty})^{k-1} } +\paren*{1+\vep^2\|\nabla u\|_{L^\infty}^2}^{k-1} \\
\times\Bigg(\int_{(\T^d)^N} \Fr_N(\ux_N,1+\vep^2 \Uu^t)d\rho_{\hbar,\vep,N}^t(\ux_N)  +  \frac{C_{d}(1+\vep^2\|\nabla u^t\|_{L^\infty}^2) (1+(\log N)\indic_{d=2})}{N^{2/d}} \Bigg)^{1/2}\Bigg).
\end{multline}
The right-hand side is evidently controlled by the quantum modulated energy $\G_{\hbar,\vep,N}(R_{\hbar,\vep,N}^t, u^t)$, and therefore this last estimate yields the desired Sobolev convergence. Convergence in the weak-* topology for $\M((\T^d)^k)$ of each of the components of $(u^t)^{\otimes k}$ follows from the preceding Sobolev convergence and the embedding $W^{-s,p'}(\T^d) \subset C(\T^d)$.

\subsection{Proof of \cref{thm:mainH}}\label{ssec:MpfH}
We now prove \cref{thm:mainH}. Our approach is similar to the proof of \cite[Proposition 2.4]{GP2021}, which yields convergence of the density and current of the nonlinear Hartree equation to a solution of the Euler-Poisson equation in the classical limit, except we emphasize the ``mean-field'' perspective in our presentation to a greater degree than in the cited work.

Let $R_{\hbar,\vep}$ denote the solution to the Hartree equation \eqref{eq:Har}. For $N\in\N$, we let $R_{\hbar,\vep}^{\otimes N}$ be the $N$-fold tensor product of $R_{\hbar,\vep}$, which yields a bosonic density matrix on the Hilbert space $\H_N$. Evidently, the k-particle marginal of $R_{\hbar,\vep}^{\otimes N}$ is just $R_{\hbar,\vep}^{\otimes k}$. Recalling the mean-field Hamiltonian $H_{\hbar,\vep}^t$ from \eqref{eq:mfham} and considering the N-body Hamiltonian
\begin{equation}
\tl{H}_{\hbar,\vep,N}^t \coloneqq \sum_{j=1}^N \Ib_{\H}^{\otimes j-1} \otimes H_{\hbar,\vep}^t \otimes \Ib_{\H}^{\otimes N-j},
\end{equation}
we see that $R_{\hbar,\vep}^{\otimes N}$ is a solution in $C(\R; \mathcal{D}_s(\H_N))$ to the Cauchy problem
\begin{equation}
\begin{cases}
i\hbar\p_t\tl{R}_{\hbar,\vep,N}^t = \comm{\tl{H}_{\hbar,\vep,N}^t}{\tl{R}_{\hbar,\vep,N}^t} \\
\tl{R}_{\hbar,\vep,N}^t|_{t=0} = (R_{\hbar,\vep}^0)^{\otimes N}.
\end{cases}
\end{equation}

Consider the modulated energy 
\begin{multline}
\G_{\hbar,\vep,N}( (R_{\hbar,\vep}^t)^{\otimes N}, u^t) = \Tr_{\H}\paren*{\sqrt{R_{\hbar,\vep}^t}|\hbar D-u^t(X)|^2\sqrt{R_{\hbar,\vep}^t}} \\
+\frac{1}{\vep^2}\int_{(\T^d)^N} \Fr_N(\ux_N,1+\vep^2\Uu^t)d(\rho_{\hbar,\vep}^t)^{\otimes N}(\ux_N).
\end{multline}
Using \cref{rem:pMEsym} and Fubini-Tonelli, we have the identity
\begin{multline}\label{eq:peidH}
\int_{(\T^d)^N}\Fr_N(\ux_N,1+\vep^2\Uu^t)d(\rho_{\hbar,\vep}^t)^{\otimes N}(\ux_N)= \frac{(N-1)}{N}\int_{\T^d}(V\ast\rho_{\hbar,\vep}^t)(x_1)d\rho_{\hbar,\vep}^t(x_1)\\
-2\int_{\T^d} (-\D)^{-1}(1+\vep^2\Uu^t)(x_1)d\rho_{\hbar,\vep}^t(x_1) + \int_{(\T^d)^2}V(x-y)d\paren*{1+\vep^2\Uu^t}^{\otimes 2}(x,y).
\end{multline}
We want to repeat the proof of \cref{prop:MEtd}, but we need to replace the identity \eqref{eq:Hrep} with
\begin{multline}
\frac{d}{dt}\Tr_{\H}\paren*{\sqrt{R_{\hbar,\vep}^t}(u^t(X)\vee \hbar D)\sqrt{R_{\hbar,\vep}^t}}=-\frac{2}{\vep^2}\int_{\T^d} u^t(x_1)\cdot \nabla (V\ast\rho_{\hbar,\ep}^t)(x_1)d\rho_{\hbar,\vep}^t(x_1) \\
+ \Tr_{\H}\paren*{ \paren*{\paren*{\frac{\hbar}{2} D\vee \nabla u^t(X) - (u^t\cdot\nabla u^t)(X) - \nabla p^t(X)}\vee\hbar D}R_{\hbar,\vep}^t}.
\end{multline}
Note that by Fubini-Tonelli and the fact that $\nabla V$ is odd, we can symmetrize the first expression in the right-hand side as
\begin{multline}
-\frac{1}{\vep^2}\int_{(\T^d)^2}\paren*{u^t(x_1)-u^t(x_2)}\cdot\nabla V(x_1-x_2)d(\rho_{\hbar,\ep}^t)^{\otimes 2}(x_1,x_2) \\
= -\frac{1}{\vep^2}\int_{(\T^d)^N}\int_{(\T^d)^2\setminus\triangle} \paren*{u^t(x)-u^t(y)}\cdot\nabla V(x-y)d\paren*{\frac{1}{N}\sum_{j=1}^N\d_{x_j}}^{\otimes 2}(x,y)d(\rho_{\hbar,\vep}^t)^{\otimes N}(\ux_N) \\
- \frac{1}{N\vep^2}\int_{(\T^d)^2}\paren*{u^t(x_1)-u^t(x_2)}\cdot\nabla V(x_1-x_2) d(\rho_{\hbar,\ep}^t)^{\otimes 2}(x_1,x_2).
\end{multline}
Ultimately, we find that
\begin{multline}\label{eq:rhsH}
\frac{d}{dt}\G_{\hbar,\vep,N}((R_{\hbar,\vep}^t)^{\otimes N}, u^t) = -2\Tr_{\H}\paren*{\sqrt{R_{\hbar,\vep}^t}\paren*{\hbar D-u^t(X)}\Sigma^t\paren*{\hbar D-u^t(X)}\sqrt{R_{\hbar,\vep}^t}}\\
+\frac{1}{\vep^2}\int_{(\T^d)^N}\int_{(\T^d)^2\setminus\triangle} \paren*{u^t(x)-u^t(y)}\cdot\nabla V(x-y)d\paren*{\frac{1}{N}\sum_{j=1}^N\d_{x_j}-1-\vep^2\Uu^t}^{\otimes 2}(x,y)d(\rho_{\hbar,\vep}^t)^{\otimes N}(\ux_N)\\
- 2\int_{(\T^d)^N}\int_{\T^d}\div(-\D)^{-1}(u^t\Uu^t)(x)d\paren*{\frac{1}{N}\sum_{j=1}^N\d_{x_j}-1-\vep^2\Uu^t}(x)d(\rho_{\hbar,\vep}^t)^{\otimes N}(\ux_N)\\
-2\int_{(\T^d)^N}\int_{\T^d}\p_t p^t(x)d\paren*{\frac{1}{N}\sum_{j=1}^N \d_{x_j}-1-\vep^2\Uu^t}(x)d(\rho_{\hbar,\vep}^t)^{\otimes N}(\ux_N) \\
+\frac{1}{N\vep^2}\int_{(\T^d)^2}\paren*{u^t(x_1)-u^t(x_2)}\cdot\nabla V(x_1-x_2) d(\rho_{\hbar,\ep}^t)^{\otimes 2}(x_1,x_2).
\end{multline}

Using the bound
\begin{equation}
\left|\paren*{u^t(x_1)-u^t(x_2)}\cdot \nabla V(x_1-x_2)\right| \lesssim_d \|\nabla u^t\|_{L^\infty}\max\{|x_1-x_2|^{2-d}, C_d\},
\end{equation}
which follows from the mean-value theorem applied to $u^t$ and \eqref{eq:Vasy}, the magnitude of the last term in the right-hand side of \eqref{eq:rhsH} is $\lesssim_d$
\begin{equation}
\frac{\|\nabla u^t\|_{L^\infty}(1+\|\rho_{\hbar,\vep}^t\|_{\dot{H}^{-1}}^2)}{N\vep^2}.
\end{equation}
We estimate the remaining terms in the right-hand side of \eqref{eq:rhsH} exactly as in \cref{ssec:Mpfgron} to conclude that there are constants $C_d,C_{d,\al}>0$ such that
\begin{multline}\label{eq:GrhsH}
\left|\frac{d}{dt}\G_{\hbar,\vep,N}((R_{\hbar,\vep}^t)^{\otimes N}, u^t)\right| \leq  C_d \|\nabla u^t\|_{L^\infty} \Tr_{\H}\paren*{\sqrt{R_{\hbar,\vep}^t}|\hbar D-u^t(X)|^2\sqrt{R_{\hbar,\vep}^t}} \\
+ \frac{C_d\|\nabla u^t\|_{L^\infty}}{\vep^2}\paren*{C_d(1+\vep^2\|\nabla u^t\|_{L^\infty}^2)\paren*{\frac{1+(\log N)\indic_{d=2}}{N^{2/d}}} + \int_{(\T^d)^N} \Fr_N(\ux_N,1+\vep^2\Uu^t)d(\rho_{\hbar,\vep}^t)^{\otimes N}(\ux_N)} \\
+C_{d,\al}\paren*{\vep^2\|u^t\|_{C^{1,\al}}^3+N^{-1/d}}\|u^t\|_{C^{1,\al}}^3 + \frac{C_d(1+\vep^2\|\nabla u^t\|_{L^\infty}^2)(1+(\log N)\indic_{d=2})}{\vep^2 N^{2/d}} \\
+ \frac{1}{\vep^{2}}\int_{(\T^d)^N}\Fr_N(\ux_N,1+\vep^2\Uu^t)d(\rho_{\hbar,\vep}^t)^{\otimes N}(\ux_N) + \frac{C_d\|\nabla u^t\|_{L^\infty}(1+\|\rho_{\hbar,\vep}^t\|_{\dot{H}^{-1}}^2)}{N\vep^2}.
\end{multline}
From the identity \eqref{eq:peidH} together with the definition \eqref{eq:ME1def} of the modulated energy, we see that
\begin{multline}\label{eq:GlimNH}
\lim_{N\rightarrow\infty} \G_{\hbar,\vep,N}((R_{\hbar,\vep}^t)^{\otimes N}, 1+\vep^2\Uu^t) \\
=  \Tr_{\H}\paren*{\sqrt{R_{\hbar,\vep}^t}|\hbar D-u^t(X)|^2\sqrt{R_{\hbar,\vep}^t}} +\frac{1}{\vep^2} \left\|\rho_{\hbar,\vep}^t-1-\vep^2\Uu^t\right\|_{\dot{H}^{-1}}^2 = \G_{\hbar,\vep}(R_{\hbar,\vep}^t,u^t).
\end{multline}
Integrating both sides of inequality \eqref{eq:GrhsH}, using the fundamental theorem of calculus, and finally letting $N\rightarrow\infty$ in the resulting right-hand side, we see that there are constants $C_d,C_{d,\al}>0$ such that 
\begin{multline}
\G_{\hbar,\vep}(R_{\hbar,\vep}^t, u^t) \leq \G_{\hbar,\vep}(R_{\hbar,\vep}^0, u^0) + C_{d,\al}\vep^2\int_0^t \|u^\tau\|_{C^{1,\al}}^6d\tau \\
+ C_d\int_0^t\paren*{1+\|\nabla u^\tau\|_{L^\infty}}\G_{\hbar,\vep}(R_{\hbar,\vep}^\tau,u^\tau)d\tau.
\end{multline}
Now applying the Gronwall-Bellman lemma, we conclude the inequality
\begin{equation}\label{eq:GgronH}
\G_{\hbar,\vep}(R_{\hbar,\vep}^t,u^t) \leq \paren*{\G_{\hbar,\vep}(R_{\hbar,\vep}^0, u^0) + C_{d,\al}\vep^2\int_0^t \|u^\tau\|_{C^{1,\al}}^6d\tau} e^{C_d \int_0^t(1+\|\nabla u^\tau\|_{L^\infty})d\tau}.
\end{equation}

From the estimate \eqref{eq:GgronH}, we deduce the convergence $(\rho_{\hbar,\vep}^t, J_{\hbar,\vep}^t)\rightarrow (1, u^t)$ as $\hbar+\vep\rightarrow 0^+$, in both the strong sense of Sobolev norms and the weak sense of measures, by repeating the arguments in \cref{ssec:Mpfmain}. This last step then completes the proof of \cref{thm:mainH}.

\bibliographystyle{alpha}
\bibliography{PointVortex}
\end{document}